\documentclass{article}%

\usepackage{amsmath}%
\usepackage{amsfonts}%
\usepackage{amssymb}%
\usepackage{amsthm}
\usepackage{graphicx}
\usepackage{a4wide}
\usepackage{todonotes}
\usepackage[all,cmtip]{xy}
\usepackage{ytableau}
\usepackage{hyperref}
\usepackage[utf8]{inputenc}
\usepackage{enumitem}

\newtheorem{theorem}{Theorem}
\numberwithin{theorem}{section}

\newtheorem{corollary}[theorem]{Corollary}

\newtheorem{lemma}[theorem]{Lemma}

\newtheorem{proposition}[theorem]{Proposition}

\theoremstyle{definition}

\newtheorem{remark}[theorem]{Remark}

\newcommand{\Hg} {\operatorname{Hg}}
\newcommand{\MT} {\operatorname{MT}}
\newcommand{\GO} {\operatorname{GO}}
\newcommand{\GL} {\operatorname{GL}}
\newcommand{\SL} {\operatorname{SL}}
\newcommand{\CSpin} {\operatorname{CSpin}}
\newcommand{\Spin} {\operatorname{Spin}}

\newcommand{\abGal}[1] {\operatorname{Gal}\big(\overline{#1}/#1\big)}

\begin{document}

\title{Non-isogenous abelian varieties sharing the same division fields}
\author{Davide Lombardo}
\date{}
\maketitle

\begin{abstract}
Two abelian varieties $A_1, A_2$ over a number field $K$ are \textit{strongly iso-Kummerian} if the torsion fields $K(A_1[d])$ and $K(A_2[d])$ coincide for all $d \geq 1$.
For all $g \geq 4$ we construct geometrically simple, strongly iso-Kummerian $g$-dimensional abelian varieties over number fields that are not geometrically isogenous. We also discuss related examples and put significant constraints on any further iso-Kummerian pair.
\end{abstract}

\section{Introduction}
The present note is motivated by the following question of Kowalski (see \cite[Problem 1.2 and Remark 3.5]{MR1981599}).
Let $A_1, A_2$ be two abelian varieties over a number field $K$ and suppose that
\begin{equation}\label{eq:Question}
K(A_1[d]) = K(A_2[d]) \text{ holds for all }d \geq 1.
\end{equation}
What can be said about $A_1, A_2$? A related question was also asked by Zywina during his talk at the conference `Abelian Varieties \& Galois Actions', held at the University of Pozna\'n in 2017.

Inspired by the terminology in \cite{MR1981599}, we say that two abelian varieties $A_1, A_2$ defined over a common number field $K$ are \textit{strongly iso-Kummerian} if condition \eqref{eq:Question} is satisfied.
Our purpose is to provide examples of strongly iso-Kummerian abelian varieties $A_1, A_2$ that share no common isogeny factors (even over $\overline{K}$), and to show that in many cases such examples are explained by the presence of a CM field in the endomorphism algebra.
For the first of these objectives, we show how a construction based on the theory of complex multiplication yields examples of strongly iso-Kummerian abelian varieties of almost all positive dimensions for which \eqref{eq:Question} is not explained by isogenies:
\begin{theorem}[Section \ref{sect:CMCounterexamples}]\label{thm:MainIntro}
Let $g \geq 4$ be an integer. There exist infinitely many (geometrically non-isomorphic) pairs of strongly iso-Kummerian abelian varieties $A_1/K, A_2/K$, where $K$ is a number field, with $A_1, A_2$ geometrically simple and not isogenous over $\overline{K}$, and such that $\dim(A_1)=\dim(A_2)=g$.
\end{theorem}

The condition $g \geq 4$ is sharp: we show in Corollary \ref{cor:DimensionAtMost3} that all strongly iso-Kummerian pairs with $A_1, A_2$ of dimension at most 3 are explained by isogenies, independently of whether they admit complex multiplication or not.
We also analyse in detail a non-simple example, originally considered by Shioda \cite{Shioda81algebraiccycles} in a slightly different context, which gives two fairly concrete strongly iso-Kummerian abelian varieties over $\mathbb{Q}$ with different isogeny factors over $\overline{\mathbb{Q}}$ (notice that in this example $\dim A_1=4, \dim A_2=3$):
\begin{theorem}[{Theorem \ref{thm:Shioda}}]
Let $J/\mathbb{Q}$ be the Jacobian of the smooth projective curve $C$ with hyperelliptic equation $y^2=x^9-1$. Then $J$ is isogenous over $\mathbb{Q}$ to $X \times E$, where $X$ is an absolutely simple abelian threefold with CM by $\mathbb{Q}(\zeta_9)$ and $E$ is the CM elliptic curve $y^2=x^3-1$. The abelian varieties $X \times E$ and $X$ are strongly iso-Kummerian over $\mathbb{Q}$.
\end{theorem}
In Section \ref{sect:EndZ} we also point out how the Kuga-Satake construction may be used to obtain `near miss' examples, where the two sides of the desired equality \eqref{eq:Question} differ at most by quadratic extensions (see Corollary \ref{cor:AlmostStronglyIsokummerian} for a precise statement).

The problem of characterising \textit{all} strongly iso-Kummerian pairs is of course much subtler.
One would wish to understand to which extent the above examples are representative, and to investigate what restrictions strongly iso-Kummerian (and non-isogenous) abelian varieties should satisfy. 
We make some partial progress in this direction by recasting condition \eqref{eq:Question} in the language of Galois representations.
Notice that the Chinese remainder theorem implies that it suffices to check \eqref{eq:Question} when $d$ is a prime power: this suggests to work with the $\ell$-adic representations attached to $A_1, A_2$ for a fixed prime $\ell$, and to rephrase the question in terms of $\ell$-adic monodromy groups. We carry this out in Section \ref{sect:PartialClassification}. This reformulation makes it clear that -- at least conditionally on the Mumford-Tate conjecture -- the problem of classifying 
non-trivial strongly iso-Kummerian pairs is closely related to the question of describing all pairs of
non-isogenous abelian varieties $A_1, A_2$ for which the Hodge group of the product $A_1 \times A_2$ is isomorphic to the Hodge group of each of its factors (see Lemma \ref{lemma:IsoKummerianImplesSameHodgeDimension} and Proposition \ref{prop:MTIsomorphic} for details).
This is a very restrictive condition, as one generically expects the Hodge group of a product to be isomorphic to the product of the Hodge groups of the factors, unless the two factors are isogenous.
In particular, in \cite{MR3494170} the author has given sufficient criteria for the equality $\Hg(A_1 \times A_2) \cong \Hg(A_1) \times \Hg(A_2)$ to hold: the fact that these criteria have wide applicability explains to a certain extent the difficulty of finding non-isogenous abelian varieties that satisfy \eqref{eq:Question}. Based on the results of \cite{MR3494170}, we extend work of Kowalski concerning iso-Kummerian elliptic curves \cite[Theorem 3.4 (4), (5)]{MR1981599} to show that for a large class of abelian varieties all strongly iso-Kummerian pairs are explained by isogenies:
\begin{theorem}\label{thm:IntroClassification}
Let $A_1, A_2$ be strongly iso-Kummerian abelian varieties over a number field $K$. Write $(A_1)_{\overline K} \sim B_1^{e_1} \times \cdots \times B_h^{e_h}$ and $(A_2)_{\overline{K}} \sim C_1^{f_1} \times \cdots \times C_k^{f_k}$ for the $\overline{K}$-isogeny decompositions of $A_1, A_2$ into simple factors. Suppose that one of the following holds:
\begin{enumerate}
\item each simple isogeny factor $B_i, C_j$ has dimension 1 or 2;
\item each simple isogeny factor $B_i, C_j$ is of type I, II or III in the sense of Albert and has odd relative dimension\footnote{the relative dimension of an abelian variety is an integer dividing its dimension, so this condition is in particular satisfied by all abelian varieties of odd dimension (provided they are of type I, II or III in the Albert classification).} (see \cite[Definition 1.3]{MR3494170}).
\end{enumerate}
Then $h=k$, and, up to renumbering, $B_i$ is isogenous to $C_i$ for each $i=1,\ldots,k$.
\end{theorem}
By definition, all (geometrically) simple abelian varieties with an action of an order in a CM field are of type IV in the classification of Albert. Thus, the combination of Theorems \ref{thm:MainIntro} and \ref{thm:IntroClassification} suggests that most strongly iso-Kummerian pairs not explained by isogenies are indeed connected to the presence of a CM field inside the endomorphism algebra of $A_1, A_2$. We also remark that the criteria of \cite{MR3494170}, especially Theorem 1.1 in op.~cit., would give further extensions of Theorem \ref{thm:IntroClassification}, but these are too technical to state concisely.

\smallskip

The paper is organised as follows. In Section \ref{sect:Preliminaries} we collect some necessary preliminaries on Hodge groups and Galois representations attached to abelian varieties, and introduce some notation relevant to the theory of complex multiplication. We also prove a result that will be used in Section \ref{sect:EndZ} to construct pairs of abelian varieties over number fields with a prescribed Mumford-Tate group (Proposition \ref{prop:DescentToNumberFields}). In Section \ref{sect:PartialClassification} we describe constraints satisfied by all strongly iso-Kummerian pairs over number fields, and in particular we connect condition \eqref{eq:Question} with properties of the Hodge and Mumford-Tate groups of $A_1, A_2$. This allows us to prove Theorem \ref{thm:IntroClassification}. In Section \ref{sect:CMMT} we describe the Mumford-Tate group of a product of simple CM abelian varieties in terms of linear algebra. In Section \ref{sect:CMCounterexamples} we prove Theorem \ref{thm:MainIntro}, discuss CM abelian varieties of dimension at most 3, and give an isolated example of a strongly iso-Kummerian pair over $\mathbb{Q}$. Finally, in Section \ref{sect:EndZ} we describe an interesting family of pairs of non-isogenous abelian varieties that have no nontrivial endomorphisms, but still come very close to satisfying \eqref{eq:Question}.

\medskip

\noindent\textbf{Acknowledgments.} I thank David Zywina for interesting discussions and Giovanni Gaiffi for providing a reference for the properties of Spin groups. I'm grateful to Heidi Goodson for pointing me to \cite{ZywinaMonodromy} and for motivating me to work out the details of Theorem \ref{thm:Shioda}. Finally, I thank Nicolas Ratazzi, who originally suggested to me the problem of studying the Galois representations attached to a product of abelian varieties.

\section{Preliminaries}\label{sect:Preliminaries}

\subsection{The Hodge group}
We briefly review the notion of Hodge group of a complex abelian variety $X$, referring the reader to \cite{Moonen_noteson} for more details. The $\mathbb{Q}$-vector space $V(X):=H_1 \left(X(\mathbb{C}),\mathbb{Q}\right)$ is naturally endowed with a Hodge structure of type $(-1,0)\oplus(0,-1)$, that is, a decomposition of $\mathbb{C}$-vector spaces $V(X) \otimes \mathbb{C} \cong V(X)^{-1,0} \oplus V(X)^{0,-1}$ such that $\overline{V(X)^{-1,0}}=V(X)^{0,-1}$.
Let $\mu_\infty:\mathbb{G}_{m,\mathbb{C}} \to \GL _{V(X)_\mathbb{C}}$ be the unique cocharacter through which $z \in \mathbb{C}^* = \mathbb{G}_{m,\mathbb{C}}(\mathbb{C})$ acts as multiplication by $z$ on $V(X)^{-1,0}$ and trivially on $V(X)^{0,-1}$. The \textit{Mumford-Tate} group $\MT(X)$ of $X$ is the $\mathbb{Q}$-Zariski closure of the image of $\mu_\infty$, that is to say, the smallest $\mathbb{Q}$-algebraic subgroup $\MT(X)$ of $\GL_{V(X)}$ such that $\mu_\infty$ factors through $\MT(X)_\mathbb{C}$.
The \textit{Hodge group} $\Hg(X)$ of $X$ is the identity component of $\MT(X) \cap \operatorname{SL}_{V(X)}$.
The group $\MT(X)$ contains the torus of homotheties in $\GL_{V(X)}$, which implies both the equality $\dim \MT(X) = \dim \Hg(X) +1$ and the fact that $\Hg(X), \MT(X)$ determine each other.
The group $\Hg(X)$ is connected and reductive, and the $\Hg(X)$-equivariant endomorphisms of the rational vector space $V(X)$ are canonically isomorphic to $\operatorname{End}^0\left(X \right) := \operatorname{End}\left(X \right) \otimes_{\mathbb{Z}} \mathbb{Q}$.
Finally, the next proposition describes the behaviour of $\Hg(X)$ with respect to products: 
\begin{proposition}\label{prop_Products} 
\begin{enumerate}
\item Let $X_1,X_2$ be complex abelian varieties. The group $\Hg(X_1 \times X_2)$ is contained in $\Hg(X_1) \times \Hg(X_2)$ and projects surjectively onto each factor. 
\item 
Let $X_1, \ldots, X_k$ be simple complex abelian varieties that are pairwise non-isogenous, and let $n_1,\ldots,n_k$ be positive integers. The groups $\Hg(X_1^{n_1} \times \cdots \times X_k^{n_k})$ and $\Hg(X_1 \times \cdots \times X_k)$ are isomorphic.
\end{enumerate}
\end{proposition}

\subsection{Galois representations}\label{subsect:GalReps}
Let $A$ be an abelian variety defined over a number field $K$. For every prime $\ell$ and every $n \geq 1$ there is a natural continuous action of $\abGal{K}$ on $A[\ell^n]$, giving rise to a representation
\[
\rho_{\ell^n} : \abGal{K} \to \operatorname{Aut} (A[\ell^n]);
\]
the extension $K(A[\ell^n])/K$ is Galois, and its Galois group can be identified with the image $G_{\ell^n}$ of $\rho_{\ell^n}$. Taking the inverse limit of this system of representations gives rise to the (continuous) $\ell$-adic representation on the Tate module $T_\ell A$,
\[
\rho_{\ell^\infty} : \abGal{K} \to \operatorname{Aut} (T_\ell A).
\]
Setting $V_\ell(A)=T_\ell(A) \otimes \mathbb{Q}_\ell$ we obtain in particular a representation $\rho_{\ell^\infty} : \abGal{K} \to \GL(V_\ell(A)) \cong \operatorname{GL}_{2\dim A}(\mathbb{Q}_\ell)$. We will call the Zariski closure of the image $G_{\ell^\infty}$ of $\rho_{\ell^\infty}$ the \textit{algebraic monodromy group at $\ell$}, and we will denote it by $G_\ell(A)$. As in the Hodge-theoretical situation, it is known that $G_\ell(A)$ contains the homotheties \cite{Bogomolov}. %
In analogy to the complex Hodge group we then define $H_\ell(A)$ as the identity component of $G_\ell(A) \cap \SL_{V_\ell(A)}$.

Suppose now that we have fixed an embedding $K \hookrightarrow \mathbb{C}$, so that we can consider $A$ as a complex abelian variety and speak of its Hodge group. The Mumford-Tate conjecture predicts that (for any choice of the embedding $K \hookrightarrow \mathbb{C}$) the group $H_\ell(A)$ should be obtained from $\Hg(A)$ by base-change from $\mathbb{Q}$ to $\mathbb{Q}_\ell$ (equivalently, that $\MT(A) \times_{\mathbb{Q}} \mathbb{Q}_\ell \cong G_\ell(A)^0$). Notice that this statement makes sense since by the comparison isomorphism of étale cohomology we can identify $V_\ell(A)$ with $H_1(A(\mathbb{C}), \mathbb{Q}) \otimes_{\mathbb{Q}} \mathbb{Q}_\ell$. The groups $\Hg(A)$ and $H_\ell(A)$ are known to share many properties. We will only use some basic ones: the group $H_\ell(A)$ is unchanged by finite extensions of the base field $K$, and if $A, B$ are $K$-abelian varieties that are $\overline{K}$-isogenous we have $H_\ell(A) \cong H_\ell(B)$. The groups $H_\ell$ also satisfy an exact analogue of Proposition \ref{prop_Products}.

\subsection{CM fields and CM types}\label{subsec:CMTypes}
A number field $E$ is said to be a \textit{CM field} if it is a totally imaginary quadratic extension of a totally real number field $E_0$. Denote by $\Sigma_E := \operatorname{Hom}_{\text{fields}}(E,\mathbb{C})$ the set of field embeddings of $E$ into $\mathbb{C}$. We briefly recall the notions of CM type, of reflex type, and of reflex norm; we refer the reader to \cite[§3]{MR608640} and \cite{MR3766118} for further details. Let $L/\mathbb{Q}$ be the Galois closure of $E/\mathbb{Q}$, let $G$ be the Galois group of $L/\mathbb{Q}$, and let $H < G$ be the subgroup corresponding to $E$. We think of $G$ as acting on $L$ on the right, and identify $\Sigma_E$ with $H \backslash G$. Any complex conjugation determines the same automorphism of $L$, hence we have a well-defined (central) involution $\rho$ of $G$ that sends $E$ to itself. For $g \in G$ we denote by $\overline{g} \in G$ the composition $g \circ \rho$, and similarly for $\varphi \in \Sigma_E$ we write $\overline{\varphi}$ for the composition $\varphi \circ \rho$.
A \textit{CM type} of $E$ is a subset $\Phi$ of $\Sigma_E = H \backslash G$ such that $\Phi \cap \overline{\Phi} = \emptyset$ and $\Phi \cup \overline{\Phi} = \Sigma_E$, where $\overline{\Phi} = \{ \overline{\phi} : \phi \in \Phi \}$. In other words, a CM type is the choice of a distinguished element in each pair of complex conjugate embeddings of $E$ into $\mathbb{C}$. Given a CM type $\Phi \subset H \backslash G$ we let $\tilde{\Phi} = \{ \varphi \in G: H\varphi \in \Phi \}$. The CM type $\Phi$ is called \textit{simple} or \textit{primitive} if $H = \{ g \in G : g \tilde{\Phi} = \tilde{\Phi} \}$. The reflex group of $\Phi$ is $H' = \{g \in G : \tilde{\Phi}g = \tilde{\Phi}\}$ and the reflex field $E^*$ is the fixed field of $H'$. The reflex CM type $\Phi^*$ of the reflex field $E^*$ is obtained as follows: letting $\tilde{\Phi}^{-1} = \{ g^{-1} : g \in \tilde{\Phi} \}$, the CM type $\Phi^*$ is the image of $\tilde{\Phi}^{-1}$ in $H' \backslash G = \Sigma_{E^*}$.
The \textit{reflex norm} corresponding to $\Phi$ is the map%
\[
\begin{array}{cccc}
N_{\Phi^*} : & (E^*)^\times & \to & E^\times \\
& x & \mapsto & \prod_{\phi \in \Phi^*} x^\phi.
\end{array}
\]
We will also say that two CM types $\Phi_1, \Phi_2$ for a CM field $E$ are \textit{essentially equal} if there exists an automorphism $\sigma$ of $E$ such that $\Phi_2$ coincides with $\Phi_1 \circ \sigma := \{\phi \circ \sigma : \phi \in \Phi_1 \}$. Otherwise we say that $\Phi_1, \Phi_2$ are \textit{essentially different}.

Let $A/K$ be a $g$-dimensional absolutely simple abelian variety, admitting complex multiplication over $\overline{K}$ by the CM field $E$, of degree $2g$ over $\mathbb{Q}$ (that is, we have $\operatorname{End}^0_{\overline{K}}(A) \cong E$). We will simply say that $A$ is a CM abelian variety, or more precisely that it has CM by $E$. The tangent space at the identity of $A_{\overline{K}}$ is both a $\overline{K}$-module and an $E$-module, and the two actions are compatible: it follows that this tangent space is a $(E \otimes \overline{K})$-bimodule, so it decomposes as $T_{\operatorname{id}} A_{\overline{K}} \cong \prod_{\varphi \in \Phi} \overline{K}_{\varphi}$, where $\overline{K}_{\varphi}$ is a 1-dimensional $\overline{K}$-vector space on which $E$ acts through the embedding $\varphi:E \hookrightarrow \overline{K}$. The set $\Phi$ of embeddings that appear in this decomposition can be shown to be a simple CM type for $E$, and we say that $A$ admits complex multiplication by the CM type $(E,\Phi)$.

\begin{remark}\label{rmk:NotIsogenous}
Notice that (with our conventions) the CM type of a CM abelian variety is only well-defined up to the choice of the isomorphism $E \cong \operatorname{End}^0_{\overline{K}}(A)$; equivalently, it is only well-defined up to
automorphisms of $E$. Two geometrically simple abelian varieties $A_1, A_2$ with CM by the same field $E$ are geometrically isogenous if and only if their respective CM types are essentially equal, see \cite[Corollary 3.13]{milneCM}.

\end{remark}

The following result is well-known, see for example \cite[Chapter I, §4 and Chapter V, Proposition 1.1]{MR713612}:
\begin{theorem}\label{thm:CMExistence}
Let $E$ be a CM field and let $\Phi$ be a simple CM type of $E$. There exist a number field $K$ and a geometrically simple CM abelian variety $A/K$ such that $\operatorname{End}^0_{\overline{K}}(A) \cong E$ and the CM type of $A$ is essentially equal to $\Phi$.
\end{theorem}

\subsubsection{The reflex norm as a map on tori}
For any number field $E$ let $T_E := \operatorname{Res}_{E/\mathbb{Q}}(\mathbb{G}_m)$ be the corresponding algebraic torus over $\mathbb{Q}$. We have in particular $T_E(\mathbb{Q})=E^\times$. The torus $T_E$ can equivalently be described as the unique $\mathbb{Q}$-torus with character group $\mathbb{Z}[\Sigma_E]$, where $\Sigma_E$ is the set of embeddings of $E$ in $\overline{\mathbb{Q}}$, endowed with its natural Galois action. In the context of CM types, one can easily check that the reflex norm $(E^*)^\times = T_{E^*}(\mathbb{Q}) \to T_E(\mathbb{Q}) = E^\times$ is induced by a map of algebraic tori $N_{\Phi^*} : T_{E^*} \to T_E$. Denoting by $\widehat{T}$ the character group of the torus $T$, the map on character groups corresponding to $N_{\Phi^*}$ is
\begin{equation}\label{eq:ReflexNorm}
\begin{array}{ccc}
\widehat{T}_{E} & \to & \widehat{T}_{E^*} \\

[\varphi] & \mapsto & \sum_{\psi \in \Phi^*} [\psi \varphi],
\end{array}
\end{equation}
where $[\varphi]$ denotes the character of $T_E$ corresponding to the embedding $\varphi : E \to \overline{\mathbb{Q}}$. It can be checked that this map is Galois-equivariant and therefore the reflex norm (seen as a map of algebraic tori) is defined over $\mathbb{Q}$.

\subsection{Galois representations in the CM case}
Let $A$ be a CM abelian variety defined over a number field $K$. We briefly review the well-known description of the Galois representations attached to $A$ in the language of class field theory. Recall that the idèles $I_K$ of a number field $K$ are the restricted product $\prod'_{v \in \Omega_K} K_v^\times$, where $\Omega_K$ is the set of all (finite and infinite) places of $K$ and the restricted product is taken with respect to the subgroups $U_v$ of local units (so that $U_v = K_v^\times$ if $v$ is infinite and $U_v = \mathcal{O}_{K_v}^\times$ if $v$ is finite). For every number field $F$ we denote by $F_\ell$ the algebra $F \otimes \mathbb{Q}_\ell$, and for an idèle $a \in I_K$ we write $a_\ell$ for the component of $a$ in $K_\ell\cong \displaystyle \prod_{\operatorname{char}(v)=\ell} K_v^\times$.
If $L/F$ is an extension of number fields, for each prime number $\ell$ the the norm map $N_{L/F} : L^\times \to F^\times$ induces local norms $N_{L/F} : L_\ell^\times \to F_\ell^\times$.
The following result, which combines \cite[Chapter 3, Theorem 1.1]{MR713612} with \cite[Theorems 6, 10 and 11, Corollary 2 to Theorem 5, and Remark on page 502]{MR0236190}, gives a description of the Galois representations in the CM case. Note that the results quoted are actually more general, but for simplicity we only state the version we will need.

\begin{theorem}\label{thm:CM}
Let $E$ be a CM field and let $A$ be an absolutely simple CM abelian variety over a number field $K$ with $\operatorname{End}_{\overline{K}}(A)\cong \operatorname{End}_K(A)\cong \mathcal{O}_E$. Let $\Phi$ be the CM type of $A$. The reflex field $E^*$ of $(E,\Phi)$ is contained in $K$, and the following hold for each prime $\ell$:
\begin{itemize}
\item $T_\ell(A)$ is a free $(\mathcal{O}_E \otimes \mathbb{Z}_\ell)$-module of rank 1;
\item the $\ell$-adic Galois representation $\rho_{\ell^\infty}$ attached to $A$ lands in $\operatorname{Aut}_{\mathcal{O}_E \otimes \mathbb{Z}_\ell}T_\ell(A) \cong (\mathcal{O}_E\otimes \mathbb{Z}_\ell)^\times$;
\item the representation $\rho_{\ell^\infty} : \operatorname{Gal}(\overline{K}/K) \to (\mathcal{O}_E\otimes \mathbb{Z}_\ell)^\times$ factors through $\operatorname{Gal}(\overline{K}/K)^{\operatorname{ab}}$, so it induces a map from $I_K$ to $(\mathcal{O}_E\otimes \mathbb{Z}_\ell)^\times$ (which we still denote by $\rho_{\ell^\infty}$). There is a unique continuous homomorphism $\varepsilon : I_K \to E^\times$ such that
\[
\rho_{\ell^\infty}(a) = \varepsilon(a) N_{\Phi^*}(N_{K/E^*}(a_\ell))^{-1} \quad \forall a \in I_K.
\]
\item let $v$ be a finite place of $K$. The restriction of $\varepsilon$ to $U_v$ has image contained in $\mu(E)$, the group of roots of unity in $E$, and is trivial if and only if $v$ is a place of good reduction of $A$.
\end{itemize}
\end{theorem}
Combined with the Néron-Ogg-Shafarevich criterion, this also has the following well-known consequence:
\begin{corollary}\label{cor:GoodReduction}
In the setting of the previous theorem, there exists a finite extension $K'$ of $K$ such that $A_{K'}$ has good reduction at all finite places of $K'$.
\end{corollary}
Finally we make the following observation, that is used implicitly both in \cite{MR3766118} and in \cite{MR1971250}:
\begin{remark}\label{rmk:EpsilonTrivialAtInfinity}
With notation as in Theorem \ref{thm:CM}, let $v$ be an infinite place of $K$. The restriction of the homomorphism $\varepsilon$ to $K_v^\times$ (considered as the subgroup of idèles $(a_w)$ with $a_w=1$ for $w \neq v$) is trivial. Indeed, since $K$ contains the CM field $E^*$, all its infinite places are complex, so that $K_v^\times \cong \mathbb{C}^\times$ is a divisible group. The image $\varepsilon(K_v^\times)$ is then a divisible subgroup of $E^\times$, where $E$ is a number field, and is therefore trivial.
\end{remark}

\subsection{Descending from $\mathbb{C}$ to number fields}\label{sect_Descent}

In Section \ref{sect:EndZ} we will need to construct abelian varieties defined over a number field with given $\ell$-adic monodromy groups; more precisely, we will need to control their $\ell$-adic monodromy in terms of corresponding Mumford-Tate groups. This motivates our interest in the following result:

\begin{proposition}\label{prop:DescentToNumberFields}
Let $A,B$ be complex abelian varieties. There exist abelian varieties $A', B'$, defined over a common number field $K$, such that the Mumford-Tate conjecture holds for $A', B', A' \times B'$, and such that $A'$ (resp.~$B'$, resp.~$A' \times B'$) has the same dimension, Hodge group, and endomorphism algebra as $A$ (resp.~$B$, resp.~$A \times B$).
\end{proposition}

Proposition \ref{prop:DescentToNumberFields} is a minor variation of the following result, due to Serre and Noot:

\begin{theorem}[{Serre, Noot \cite{MR1355123}}]\label{thm_SerreNoot}
Let $G$ be the Mumford-Tate group of an Abelian variety $X$ over $\mathbb{C}$. Then there exists an Abelian variety $Y$, defined over a number field $K$, such that $\dim(Y)=\dim(X)$ and $\MT(Y) \cong \MT(X)$.
Furthermore, if $\ell$ is a fixed rational prime, then we may choose $Y/K$ so that the image $G_\ell$ of $\operatorname{Gal}(\overline{K}/K)$ in $\operatorname{Aut}(T_\ell Y)$ is open and Zariski dense in $G$; in particular, $G_\ell(Y) = G \otimes \mathbb{Q}_\ell$ and the Mumford-Tate conjecture holds for $Y$.
\end{theorem}

We shall also need the following fact, whose proof is immediate:
\begin{lemma}
If the Mumford-Tate conjecture holds for a product $A \times B$, then it holds for both $A$ and $B$. 
\end{lemma}

\begin{proof}[Proof of Proposition \ref{prop:DescentToNumberFields}]
By applying Theorem \ref{thm_SerreNoot} to $X=A \times B$ and to any prime $\ell$ we get an abelian variety $Y$ defined over a number field $K$, satisfying the Mumford-Tate conjecture, and such that $\MT(Y) \cong \MT(X)$ (hence $\Hg(Y) \cong \Hg(X)$). From the proof of \cite[Theorem 1.7]{MR1355123} we see that $V_Y := H_1(Y(\mathbb{C}), \mathbb{Q})$ and $V_X := H_1(X(\mathbb{C}), \mathbb{Q})$ are isomorphic as representations of $\MT(Y) \cong \MT(X)$. In particular, $\operatorname{End}^0_{\overline{\mathbb{Q}}}(Y)  \cong V_Y^{\Hg(Y)} \cong V_X^{\Hg(X)} \cong \operatorname{End}_{\overline{\mathbb{Q}}}^0(X)$. Let $\pi_A, \pi_B \in \operatorname{End}_{\overline{\mathbb{Q}}}^0(X)$ be the projectors on the factors $A, B$. Under the above isomorphism, these correspond to nontrivial orthogonal idempotents $\pi_{A'}, \pi_{B'} \in \operatorname{End}^0_{\overline{\mathbb{Q}}}(Y)$, which shows that $Y$ is isogenous to a product $A' \times B'$. Moreover, the Hodge groups of $A'$ and $B'$ can be described as $\pi_{A'}( \Hg(Y)) \cong \pi_A (\Hg(X)) = \Hg(A)$, $\pi_{B'}( \Hg(Y)) \cong \pi_B (\Hg(X)) = \Hg(B)$. Since the defining representation of the Hodge group determines the dimension and endomorphism algebra of an abelian variety, we obtain that $A$ and $A'$ (resp.~$B$ and $B'$) have the same dimension and isomorphic endomorphism algebras.
Finally, the previous lemma shows that the Mumford-Tate conjecture holds for both $A'$ and $B'$ (since they are factors of $Y$ and Mumford-Tate holds for $Y$).
\end{proof}

\section{Constraints satisfied by strongly iso-Kummerian pairs}\label{sect:PartialClassification}
In this section we find necessary conditions that every pair of strongly iso-Kummerian abelian varieties over a number field must satisfy. This doesn't quite lead to a complete classification, but shows that non-trivial examples must be quite special. In particular, the second part of Theorem \ref{thm:IntroClassification}, that we will prove below, shows that no non-trivial examples may be found among abelian varieties all of whose simple factors have odd dimension, unless some of these factors admit an action by an order in a CM field (i.e., are of type IV in the Albert classification). Given the examples we construct in the next section, this result is thus in a sense the best possible. 

It is useful to take the following point of view: if \eqref{eq:Question} holds, then in particular we have
\[
K( (A_1 \times A_2)[\ell^n] ) = K( A_1[\ell^n] )  = K( A_2[\ell^n] )
\]
for all positive integers $n$. We now show that this implies $\dim G_\ell(A_1 \times A_2) = \dim G_\ell(A_1) = \dim G_\ell(A_2)$. This is based on the following lemma:
\begin{lemma}
Let $A$ be an abelian variety defined over a number field $K$ and let $\ell$ be a prime number. Let $d$ be the dimension of $G_\ell(A)$. There exist positive constants $c_1, c_2$, depending on $A, \ell$ and $K$, such that for all $n \geq 1$
\[
c_1 \ell^{nd} \leq [ K(A[\ell^n]) : K ] \leq c_2 \ell^{nd}.
\]
\end{lemma}
\begin{proof}
Recall that $G_{\ell^n}$ denotes the image of the modulo-$\ell^n$ representation attached to $A/K$. By \cite[Lemma 12]{lombardo_perucca_2019}, there exists $n_0$ such that for $n \geq n_0$ we have $\#G_{\ell^{n+1}} = \ell^d \#G_{\ell^n}$ for all $n \geq n_0$. Thus for sufficiently large $n$ we have $\# G_{\ell^n} =  \# G_{\ell^{n_0}} \cdot \ell^{(n-n_0)d}$. Since $[ K(A[\ell^n]) : K ] = \#G_{\ell^n}$, this concludes the proof.
\end{proof}

 Set $A=A_1 \times A_2$. As remarked above, the equality
 $K(A_1[\ell^n])=K(A_2[\ell^n])$ implies $K(A_1[\ell^n]) = K(A_2[\ell^n]) = K(A[\ell^n])$. By the previous lemma, the degrees over $K$ of these fields grow (up to a bounded factor) as $\ell^{n \dim G_\ell(A_1)}$, $\ell^{n \dim G_\ell(A_2)}$ and $\ell^{n \dim G_\ell(A)}$ respectively. As $n \to \infty$, this implies $\dim G_\ell(A) = \dim G_\ell(A_1) = \dim G_\ell(A_2)$ as claimed; the analogous equality also holds with $G_\ell$ replaced by $H_\ell$.
Since we also have canonical surjective maps $H_\ell(A) \to H_\ell(A_i)$ for $i=1, 2$, the dimension condition implies that both maps $H_\ell(A) \to H_\ell(A_i)$ have finite kernel: this is then a \textit{necessary} condition for \eqref{eq:Question} to hold. %
This discussion implies the following lemma (notice that if Mumford-Tate holds for $A_1$ and $A_2$, then it also holds for $A_1 \times A_2$ \cite{MR4009176}):
\begin{lemma}\label{lemma:IsoKummerianImplesSameHodgeDimension}
Let $A_1, A_2$ be two strongly iso-Kummerian abelian varieties over a number field $K$. Suppose that the Mumford-Tate conjecture holds for $A_1$ and $A_2$. Then $\dim \Hg(A_1 \times A_2) = \dim \Hg(A_1) = \dim\Hg(A_2)$, and the canonical projections $\Hg(A_1 \times A_2) \to \Hg(A_i)$, $\MT(A_1 \times A_2) \to \MT(A_i)$ are isogenies for $i=1,2$.
\end{lemma}
 We can now prove Theorem \ref{thm:IntroClassification}.

\begin{proof}[Proof of Theorem \ref{thm:IntroClassification}]
By symmetry, it suffices to show that the simple isogeny factor $B_i$ of $(A_1)_{\overline{K}}$ is also an isogeny factor of $(A_2)_{\overline{K}}$. Suppose by contradiction that this is not the case. Notice that over $\overline{K}$ there is a surjective homomorphism of abelian varieties $A_1 \times A_2 \to B_i \times A_2$, which gives $\dim H_\ell(A_1 \times A_2) \geq \dim H_\ell(B_i \times A_2)$. Under the assumptions of the theorem, \cite[Corollaries 4.4 and 4.5]{MR3494170} give
\[
\dim H_\ell(A_1 \times A_2) \geq \dim H_\ell(B_i \times A_2) = \dim H_\ell(B_i) + \dim H_\ell(A_2) > \dim H_\ell(A_2),
\]
contradicting the necessary condition we found above. Notice that we have used that the group $H_\ell(B_i)$ is positive-dimensional for any abelian variety $B_i$ over any number field. This follows for example from the fact that if $H_\ell(B_i)$ is trivial, then Tate's conjecture on endomorphisms (proved by Faltings) implies
\[
\operatorname{End}_{\overline{K}}(B_i) \otimes \mathbb{Q}_\ell \cong \operatorname{End}_{H_\ell(B_i)}(T_\ell B_i \otimes \mathbb{Q}_\ell)  \cong \operatorname{Mat}_{2\dim(B_i) \times 2\dim(B_i)} (\mathbb{Q}_\ell),
\]
which is impossible for dimension reasons (here we use that $K$ is of characteristic $0$).
\end{proof}

\begin{remark}
The conclusion in Theorem \ref{thm:IntroClassification} concerns the existence of isogenies over $\overline{K}$ (or equivalently, over a suitable finite extension of the ground field). Generalising \cite[Example 8.5]{MR1981599}, we show that passing to an extension is necessary: indeed, consider two abelian varieties $A, B$ over $\mathbb{Q}$ and a non-trivial quadratic character $\chi$ of $\operatorname{Gal}(\overline{\mathbb{Q}}/\mathbb{Q})$. Denote by $A^\chi$ (respectively $B^\chi$) the quadratic twist of $A$ (resp.~$B$) with character $\chi$. Finally set $A_1 = A \times A^{\chi} \times B$ and $A_2 = A \times B^{\chi} \times B$. For generic choices of $A$ and $B$, there are no non-trivial isogenies between $A, B, A^{\chi}, B^{\chi}$ defined over $\mathbb{Q}$. The modulo-$\ell^n$ Galois representation attached to $A_1$ sends $\sigma \in \abGal{\mathbb{Q}}$ to $\left( \rho_{A, \ell^n}(\sigma), \chi(\sigma) \rho_{A,\ell^n}(\sigma), \rho_{B,\ell^n}(\sigma)  \right)$, while for $A_2$ we have $\rho_{A_2, \ell^n}(\sigma) = \left( \rho_{A, \ell^n}(\sigma), \chi(\sigma) \rho_{B,\ell^n}(\sigma), \rho_{B,\ell^n}(\sigma)  \right)$. It is easy to check that $\sigma$ belongs to the kernel of $\rho_{A_1, \ell^n}$ if and only if $\rho_{A, \ell^n}(\sigma) = \operatorname{Id}, \rho_{B,\ell^n}(\sigma) = \operatorname{Id}$, and $\chi(\sigma) = +1$, and that the same conditions are also equivalent to $\sigma$ being in the kernel of $\rho_{A_2,\ell^n}$. Since the torsion fields $\mathbb{Q}(A_1[\ell^n]), \mathbb{Q}(A_2[\ell^n])$ correspond by Galois theory to the kernels of $\rho_{A_1, \ell^n}, \rho_{A_2, \ell^n}$ respectively, this shows that $\mathbb{Q}(A_1[\ell^n])=\mathbb{Q}(A_2[\ell^n])$ holds for all primes $\ell$ and all positive integers $n$. In particular, $A_1, A_2$ are strongly iso-Kummerian without having the same $\mathbb{Q}$-isogeny factors.
\end{remark}

We have shown above that strongly iso-Kummerian abelian varieties have isogenous Mumford-Tate groups. As a partial converse, we have the following proposition:
\begin{proposition}\label{prop:MTIsomorphic}
Let $A_1/K_1, A_2/K_2$ be abelian varieties defined over number fields $K_1, K_2$. Suppose that 
the canonical projections $\MT(A_1 \times A_2) \to \MT(A_i)$ are isomorphisms, for $i=1,2$. There exists a number field $L$, containing both $K_1$ and $K_2$, such that the base-changes of $A_1, A_2$ to $L$ satisfy
\[
L(A_1[\ell^\infty]) = L(A_2[\ell^\infty])
\]
for all prime numbers $\ell$.
\end{proposition}
\begin{proof}
Let $K$ be any number field containing both $K_1$ and $K_2$.  For any abelian variety $A$ over a number field $K$, by \cite[I, Proposition 6.2]{DeligneInclusion} we know that the connected component of the Zariski closure of $G_{A,\ell^\infty}$ is contained in $\MT(A)_{\mathbb{Q}_\ell}$ for all $\ell$. By \cite[p.~17]{MR1730973}, or equivalently \cite[Proposition 6.14]{MR1150604}, we also know that there is a finite extension of $K$ for which the Zariski closure of $G_{A,\ell^\infty}$ is connected for all $\ell$. Letting $L$ be such an extension for the abelian variety $A_1 \times A_2 / K$, we obtain that for all primes $\ell$ the $\ell$-adic Galois representations attached to $A_1 \times A_2, A_1$, and $A_2$ over $L$ land in the $\mathbb{Q}_\ell$-points of the respective Mumford-Tate groups.

Let $A:=A_1 \times A_2$ and denote by $\rho_{A_1, \ell^\infty}, \rho_{A_2, \ell^\infty}$ and $\rho_{A, \ell^\infty}$ the corresponding Galois representations. 
For $i=1,2$ we have a commutative diagram
\[\xymatrixcolsep{4pc}
\xymatrix{
\abGal{L} \ar[r]^-{\rho_{A,\ell^\infty}}
\ar[dr]_{\rho_{A_i,\ell^\infty}} & G_{A,\ell^\infty} \ar@{->>}[d] \ar@{^{(}->}[r] &  \MT(A)(\mathbb{Q}_\ell) \ar@{->>}[d] \\ %
 &  G_{A_i,\ell^\infty} \ar@{^{(}->}[r] &  \MT(A_i)(\mathbb{Q}_\ell). %
}
\]
By assumption, the vertical map $\MT(A)(\mathbb{Q}_\ell) \to \MT(A_i)(\mathbb{Q}_\ell)$ is injective, hence an isomorphism. By an easy diagram chasing, this implies that $\pi_i : G_{A, \ell^\infty} \to G_{A_i, \ell^\infty}$ is injective (hence, again, an isomorphism). 
It follows that $\rho_{A_i, \ell^\infty}$ and $\rho_{A,\ell^\infty}$ have the same kernel for $i=1,2$, hence that $L(A_1[\ell^\infty]) = L(A[\ell^\infty])= L (A_2[\ell^\infty])$ as claimed. 
\end{proof}

\section{Mumford-Tate groups of products of CM abelian varieties}\label{sect:CMMT}

Let $A_1, A_2$ be geometrically simple CM abelian varieties (defined over a number field, or over $\mathbb{C}$), having complex multiplication by two CM fields $E_1$ and $E_2$, with simple CM types $\Phi_1, \Phi_2$ respectively. Write $g_i = \dim A_i$.
We now explain how to describe the Mumford-Tate group of $A_1 \times A_2$ in terms of integral linear algebra, and in particular how to check whether one of the canonical projections $\MT(A_1 \times A_2) \to \MT(A_i)$ is an isogeny or an isomorphism (which, in the light of Lemma \ref{lemma:IsoKummerianImplesSameHodgeDimension} and Proposition \ref{prop:MTIsomorphic}, is closely related to the problem of deciding whether $A_1, A_2$ are strongly iso-Kummerian over some number field). It is straightforward to extend all we are going to describe to the product of an arbitrary number of CM abelian varieties, but the notation quickly becomes cumbersome, so for simplicity of exposition we limit ourselves to the case of two factors. For a related but distinct approach see also Section 6 of \cite{emory2020satotate}, which however only applies to certain hyperelliptic Jacobians.

Let $L_1, L_2$ be the Galois closures of $E_1/\mathbb{Q}, E_2/\mathbb{Q}$ respectively, and for $i=1,2$ let $H_i$ be the subgroup of $G_i = \operatorname{Gal}(L_i/\mathbb{Q})$ given by $\operatorname{Gal}(L_i/E_i)$. As explained in Section \ref{subsec:CMTypes}, the CM types $\Phi_1, \Phi_2$ can be considered as subsets of $H_i \backslash G_i$, and we obtain reflex subgroups $H_i' < G_i$, reflex fields $E_i^* \subseteq L_i$, and reflex CM types $\Phi_1^*, \Phi_2^*$. Fix a finite Galois extension $L$ of $\mathbb{Q}$ containing both $L_1$ and $L_2$ (hence, a fortiori, both $E_1^*$ and $E_2^*$) and denote by $N_i : L \to E_i^*$ the norm function. We also denote by the same letter the induced map $T_{L} \to T_{E_i^*}$ between the corresponding algebraic tori. Finally, let $\phi_i = N_{\Phi_i^*} : T_{E_i^*} \to T_{E_i}$ be the reflex norm corresponding to $\Phi_i^*$.
The following is a consequence of the theory of complex multiplication, and essentially a reformulation of part of Theorem \ref{thm:CM}:
\begin{lemma}\label{lemma:MapToMTOfProduct}
The image of the map $T_L \xrightarrow{(N_1, N_2)} T_{E_1^*} \times T_{E_2^*} \xrightarrow{(\phi_1,\phi_2)} T_{E_1} \times T_{E_2}$ is $\MT(A_1 \times A_2)$.
\end{lemma}
\begin{proof}
Fix a common (number) field of definition $K$ for $A_1$ and $A_2$. Enlarging $K$ if necessary, and replacing each $A_i$ by an isogenous variety (which does not alter their Mumford-Tate groups), we may assume that $K$ contains $L$, that $\operatorname{End}_K(A_i)=\mathcal{O}_{E_i}$ for $i=1,2$, and that $A_i$ has good reduction at all finite places of $K$. 
Let $M$ be the finite extension of $K$ corresponding to the open subgroup $K^\times \cdot \left(\prod_{v \text{ finite place of }K} U_v \times \prod_{v \text{ infinite}} K_v^\times\right)$ of $I_K$. Finally fix a prime $\ell$ for which $T_M, T_L, T_{E_i}$ and $T_{E_i^*}$ all have good reduction (any prime unramified in $M$ will work). For such a prime it makes sense to take the $\mathbb{Z}_\ell$-points of the tori $T_M, T_L, T_{E_i}$ and $T_{E_i^*}$, and
 $T_F(\mathbb{Z}_\ell) = (\mathcal{O}_F \otimes \mathbb{Z}_\ell)^\times = \prod_{\substack{v \text{ place of }F \\ \operatorname{char}(v)=\ell}} U_v$ for $F=M, L, E_i, E_i^*$.

We now use Theorem \ref{thm:CM} to describe the $\ell$-adic Galois representations attached to $A_1, A_2$ over $M$. Notice that the characters $\varepsilon_i$ are trivial: they are trivial on $U_v$ for all finite $v$ since $A_i$ has good reduction at $v$, and at all the infinite places by Remark \ref{rmk:EpsilonTrivialAtInfinity}. Thus these representations are obtained as the maps on $\mathbb{Z}_\ell$-points induced by the morphisms of algebraic tori
\[
T_M \xrightarrow{N_{M/L}} T_L \xrightarrow{N_i} T_{E_i^*} \xrightarrow{\phi_i} T_{E_i}.
\]
The simultaneous Galois action on $A_1 \times A_2$ is then the map on $\mathbb{Z}_\ell$-points induced by $T_M \xrightarrow{N_{M/L}} T_L \to T_{E_1^*} \times T_{E_2^*} \to T_{E_1} \times T_{E_2}$. Given that the Mumford-Tate conjecture holds for $A_1 \times A_2$ by \cite{MR228500}, the Mumford-Tate group of $A_1 \times A_2$ is the smallest subtorus $T$ of $T_{E_1} \times T_{E_2}$ with the property that (for all sufficiently large number fields $M$ and all primes $\ell$) the image of the $\ell$-adic Galois representation attached to $A_1 \times A_2 / M$ factors via $T(\mathbb{Q}_\ell)$. Since $T_M \xrightarrow{N_{M/L}} T_L$ is surjective, it is clear from the above description that the minimal such torus is precisely the image of the map $T_L \to T_{E_1^*} \times T_{E_2^*} \to T_{E_1} \times T_{E_2}$, as claimed. %
\end{proof}

We consider the commutative diagram (well-defined by the previous lemma)
\[\xymatrixcolsep{5pc}
\xymatrix{
T_L \ar@{->>}[r]^{(\phi_1 N_1, \phi_2 N_2)} \ar[d]^{(N_1, N_2)}  & \MT(A_1\times A_2) \ar@{->>}[r]^{\pi_i}  \ar@{^{(}->}[d] & \MT(A_i) \ar@{^{(}->}[d] \\
T_{E_1^*} \times T_{E_2^*}  \ar[r]_{(\phi_1, \phi_2)} & T_{E_1} \times T_{E_2} \ar@{->>}[r]_{p_i} & T_{E_i}
}
\]
and its dual
\[\xymatrixcolsep{5pc}
\xymatrix{
\widehat{T}_L & \widehat \MT(A_1\times A_2) \ar@{_{(}->}[l]_{N_1^* \phi_1^* + N_2^* \phi_2^*}    & \widehat \MT(A_i) \ar@{_{(}->}[l]_{\pi_i^*}  \\
\widehat T_{E_1^*} \times \widehat T_{E_2^*} \ar[u]^{N_1^* + N_2^*}   & \widehat T_{E_1} \times \widehat T_{E_2} \ar@{->>}[u] \ar[l]^{(\phi_1^*, \phi_2^*)} & \widehat T_{E_i} \ar@{_{(}->}[l]^{p_i^*} \ar@{->>}[u]
 }
\]
where $\pi_i, p_i$ are the canonical projections on the $i$-th factor, and we denote by $f^*$ the map induced on characters by a morphism $f$ of algebraic tori.%

The condition that $\pi_i : \MT(A_1 \times A_2) \to \MT(A_i)$ be an isomorphism (respectively, an isogeny) translates into $\pi_i^*$ being an isomorphism (respectively, having finite kernel and cokernel), and since $\pi_i : \MT(A_1 \times A_2) \to \MT(A_i)$ is certainly surjective (hence $\pi_i^*$ is injective), it suffices to check that $\pi_i^*$ is onto (resp.~has finite cokernel). By diagram chasing, this is equivalent to the composition $\widehat{T}_{E_i} \to \widehat{T}_{E_1} \times \widehat{T}_{E_2} \to \widehat{\MT}(A_1 \times A_2)$ being surjective (resp.~having finite cokernel).

As $\widehat{\MT}(A_1 \times A_2) \to \widehat{T}_L$ is injective, the kernel of the map $\widehat{T}_{E_1} \times \widehat{T}_{E_2} \to \widehat{\MT}(A_1 \times A_2)$ coincides with the kernel of $\widehat{T}_{E_1} \times \widehat{T}_{E_2} \xrightarrow{N_1^* \phi_1^* + N_2^*\phi_2^*} \widehat{T}_L$. Thus we may identify $\widehat{\MT}(A_1 \times A_2)$ with $\displaystyle \frac{\widehat{T}_{E_1} \times \widehat{T}_{E_2}}{ \ker (N_1^* \phi_1^* + N_2^*\phi_2^*) }$, and under this identification $\pi_i^*$ is surjective (resp.~has finite cokernel) if and only if $p_i^*\left( \widehat{T}_{E_i} \right)$  -- which is $\widehat{T}_{E_1} \times \{0\}$ or $\{0\} \times \widehat{T}_{E_2} $ according to whether $i=1$ or $2$ -- surjects onto $\displaystyle \frac{\widehat{T}_{E_1} \times \widehat{T}_{E_2}}{ \ker (N_1^* \phi_1^* + N_2^*\phi_2^*) }$ (resp.~maps to it with finite cokernel). We have thus established:
\begin{lemma}\label{lemma:MTIsoLinearAlgebra}
The canonical projection $\pi_1$, respectively $\pi_2$, is an isomorphism if and only if 
\[
\widehat{T}_{E_1} \times \{0\} + \ker (N_1^* \phi_1^* + N_2^*\phi_2^*) = \widehat{T}_{E_1} \times \widehat{T}_{E_2},
\]
respectively $
\{0\} \times \widehat{T}_{E_2} + \ker (N_1^* \phi_1^* + N_2^*\phi_2^*) = \widehat{T}_{E_1} \times \widehat{T}_{E_2},
$. Similarly, the canonical projection $\pi_1$, respectively $\pi_2$, is an isogeny if and only $
\{0\} \times \widehat{T}_{E_2} + \ker (N_1^* \phi_1^* + N_2^*\phi_2^*)$, respectively $
\widehat{T}_{E_1} \times \{0\} + \ker (N_1^* \phi_1^* + N_2^*\phi_2^*)$, has finite index in $\widehat{T}_{E_1} \times \widehat{T}_{E_2}$.
\end{lemma}

We now rephrase this lemma in terms of matrices.
Write $G$ for the Galois group of $L$ over $\mathbb{Q}$, and let $f_i : G \to G_i$ be the canonical projections. The character group $\widehat{T}_L$ is canonically isomorphic to $\mathbb{Z}[G] \cong \mathbb{Z}^{|G|}$, while $\widehat{T}_{E_i} \cong \mathbb{Z}[H_i \backslash G_i] \cong \mathbb{Z}^{2g_i}$ and $\widehat{T}_{E_i^*} \cong \mathbb{Z}[H'_i \backslash G_i] \cong \mathbb{Z}^{|H'_i \backslash G|}$. %
In order to make Lemma \ref{lemma:MTIsoLinearAlgebra} explicit, it now suffices to write down a matrix representing $N_1^* \phi_1^* + N_2^*\phi_2^*$ with respect to the bases of $\widehat{T}_{E_1} \times \widehat{T}_{E_2}$ and $\widehat{T}_L$ induced by the groups $H_i \backslash G_i$ and $G$. Clearly it suffices to do so for $N_1^*\phi_1^*$ and $N_2^*\phi_2^*$ separately, and this is achieved by noticing that $\phi_i^*$ is described by Equation \eqref{eq:ReflexNorm}, while
\[
\begin{array}{cccc}
N_i^* : & \widehat{T}_{E_i^*} & \to & \widehat{T}_L \\
& [H_i'g] & \mapsto & \sum_{f_i(h) \in H_i'} [hg].
\end{array}
\]
This determines the matrices $M_1, M_2$ representing $N_1^*\phi_1^*, N_2^* \phi_2^*$ with respect to the bases above, thus allowing Lemma \ref{lemma:MTIsoLinearAlgebra} to be used algorithmically. 
In the important special case $E_1=E_2=E$ we may take $L=L_1=L_2$, $G=G_1=G_2$ and $H_1=H_2=:H$, and the composition $N_i^* \phi_i^*$ is then described by
\[
\begin{array}{cccc}
N_i^* \phi_i^* : & \widehat{T}_E & \to & \widehat{T}_L \\
& [Hg] & \mapsto & \sum_{r \in \widetilde{\Phi_i^*}} [rg].
\end{array}
\]
The corresponding matrix $M_i$ (with rows indexed by $G$ and columns indexed by $H \backslash G$) has coefficient in position $(g_1, Hg_2)$ given by
$
\begin{cases}
1, \text{ if } g_1 g_2^{-1} \in \widetilde{\Phi_i^*} \\
0, \text{ otherwise}
\end{cases}
$.
The matrix $M$ of $N_1^*\phi_1^* + N_2^*\phi_2^*$ has rows indexed by $G$ and columns indexed by $H\backslash G \times \{1,2\}$, and is obtained by concatenating the matrices of $N_1^*\phi_1^*$, $N_2^*\phi_2^*$ horizontally. 
Lemma \ref{lemma:MTIsoLinearAlgebra} then takes the following completely explicit form:
\begin{lemma}\label{lemma:MTIsoLinearAlgebraExplicit} Suppose $A_1, A_2$ are $g$-dimensional simple CM abelian varieties with complex multiplication by the same CM field $E$.
The projection $\MT(A_1 \times A_2) \to \MT(A_1)$ is an isomorphism if and only if the $\mathbb{Z}$-submodule of $\mathbb{Z}^{2g} \oplus \mathbb{Z}^{2g} \cong \mathbb{Z}[H \backslash G] \oplus \mathbb{Z}[H \backslash G]$ spanned by $\ker M$ and by $\mathbb{Z}^{2g} \oplus \{0\}$ is all of $\mathbb{Z}^{2g} \oplus \mathbb{Z}^{2g}$. Similarly, $\MT(A_1 \times A_2) \to \MT(A_1)$ is an isogeny if and only if $\ker M$ and $\mathbb{Z}^{2g} \oplus \{0\}$ span a finite-index submodule of $\mathbb{Z}^{2g} \oplus \mathbb{Z}^{2g}$.
\end{lemma}

As an application of this circle of ideas we show the following statement concerning CM abelian threefolds:
\begin{proposition}\label{prop:CMThreefolds}
Let $A_1, A_2$ be simple CM abelian threefolds over $\mathbb{C}$ for which $\dim \MT(A_1 \times A_2) = \dim \MT(A_1) = \dim \MT(A_2)$. Then $A_1$ and $A_2$ are isogenous.
\end{proposition}
\begin{proof}
The Mumford-Tate group of a simple abelian threefold with CM by the field $E$ is equal to the $\mathbb{Q}$-torus given as a functor by
\[
U_E(R) = \{ x \in T_E(R) : xx^* \in R^\times  \}
\]
for all $\mathbb{Q}$-algebras $R$, see \cite[§2.3]{MZ}. Let $L$ be the Galois closure of $E/\mathbb{Q}$, let $G = \operatorname{Gal}(L/\mathbb{Q})$ and write $H$ for the subgroup of $G$ corresponding to $E$. The group $G$ has nontrivial centre of even order (complex conjugation is a central element of order 2) and is isomorphic to a subgroup of $S_6$ of order divisible by 6 (since it is the Galois group of the normal closure of the field $E$, of degree 6). Up to isomorphism, there are precisely 4 such groups: $\mathbb{Z}/6\mathbb{Z}, D_6, \mathbb{Z}/2\mathbb{Z} \times A_4$  and $\mathbb{Z}/2\mathbb{Z} \times S_4$. Notice that each of these groups has a single central involution $\rho$, so complex conjugation is uniquely identified by the group structure.

The rational character group $\hat{U}_E \otimes_\mathbb{Z} \mathbb{Q}$ can be identified with $\mathbb{Q}[H \setminus G]$, and in particular is equipped with the natural action of $\abGal{\mathbb{Q}}$, which is nothing but the composition of the natural projection $\abGal{\mathbb{Q}} \to G$ with the
permutation action of $G$ on the coset space $H \setminus G$. Since $L$ is the Galois closure of $E$, the core of $H$ in $G$ is trivial, hence the kernel of the permutation action is trivial. Translating this group-theoretic statement to Galois theory, we obtain that the kernel of the Galois action on $\hat{U}_E \otimes_\mathbb{Z} \mathbb{Q}$ defines the extension $L/\mathbb{Q}$, hence that the Mumford-Tate group determines $L$.

Let now $E_1, E_2$ be the fields of complex multiplication of $A_1, A_2$ respectively. The hypothesis implies that the projections $\MT(A_1 \times A_2) \to \MT(A_i)$ are isogenies, so the rational character groups
\[
\widehat{\MT}(A_1) \otimes \mathbb{Q}, \quad \widehat{\MT}(A_2) \otimes \mathbb{Q}, \quad \widehat{\MT}(A_1 \times A_2) \otimes \mathbb{Q}
\]
are all isomorphic (also as Galois modules). %
In particular, $E_1$ and $E_2$ have the same Galois closure $L/\mathbb{Q}$.
The group $G = \operatorname{Gal}(L/\mathbb{Q})$ is isomorphic to one of the finitely many possibilities listed above. For each such group, there are only finitely many possibilities for %
the subgroups $H_1, H_2$ (not containing $\rho$) corresponding to the sub-extensions $L/E_1$, $L/E_2$, and for each $H_i$ there are only finitely many possible CM types for $H_i \setminus G$. We wrote a short script \cite{Script} that loops over all $G$ in the above list, %
all subgroups $H_1, H_2$ not containing $\rho$ and satisfying $[G:H_1] = [G:H_2] = 6$, and all pairs $(\Phi_1, \Phi_2)$ of CM types of $H_1 \backslash G$, $H_2 \backslash G$. For each such possibility we have then determined the dimensions of $\MT(A_1), \MT(A_2)$ and $\MT(A_1 \times A_2)$ (this uses the identification of $\widehat{\MT}(A_1 \times A_2)$ with $\displaystyle \frac{\widehat{T}_{E_1} \times \widehat{T}_{E_2}}{ \ker (N_1^* \phi_1^* + N_2^*\phi_2^*) }$ proved above). In all cases we have thus been able to check that the equality $\dim \MT(A_1 \times A_2) = \dim \MT(A_1) = \dim \MT(A_2)$ implies that $H_1, H_2$ are conjugated in $G$ (hence that the fields $E_1, E_2$ are isomorphic), and that under the identification $E_1 \cong E_2$ induced by conjugation the CM types $\Phi_1, \Phi_2$ are essentially equal. We conclude as desired that $A_1, A_2$ are isogenous.
\end{proof}

\begin{remark}
In the context of the previous proof, we know a priori that $\dim \MT(A_i)=4$ by \cite[Examples 3.7]{MR608640}. Our computations confirm this independently.
\end{remark}

\section{Examples of strongly iso-Kummerian pairs}\label{sect_Counterexample}

\subsection{CM abelian varieties}\label{sect:CMCounterexamples}

Fix a positive integer $g$ not equal to $1, 2, 3, 4, 6$. We describe the construction of infinitely many pairs of abelian varieties $(A_1, A_2)$ with $\dim A_1=\dim A_2=g$ that satisfy \eqref{eq:Question}, are absolutely simple, and geometrically non-isogenous. 
It is easy to see that there exist infinitely many non-isomorphic totally real number fields $E_0$, Galois over $\mathbb{Q}$, with $\operatorname{Gal}(E_0/\mathbb{Q}) \cong \mathbb{Z}/g\mathbb{Z}$ (simply consider appropriate real subfields of cyclotomic fields). %
We further fix a quadratic imaginary field $F$ and let $E=E_0F$ be a CM field. %
We will tacitly identify $\operatorname{Hom}_{\text{fields}}(E,\mathbb{C})$ with $\operatorname{Gal}(E/\mathbb{Q})$. Notice that this group can be canonically identified with the product $\operatorname{Gal}(E_0/\mathbb{Q}) \times \operatorname{Gal}(F/\mathbb{Q})$, and that the generator $\iota$ of the direct factor $\operatorname{Gal}(F/\mathbb{Q})$ is complex conjugation on $E$. We also fix a generator $\psi$ of the cyclic group $\operatorname{Gal}(E/F) \cong \operatorname{Gal}(E_0/\mathbb{Q})$ and write the $2g$ elements of $\operatorname{Gal}(E/\mathbb{Q})=\operatorname{Hom}_{\text{fields}}(E,\mathbb{C})$ as $\sigma_0,\ldots,\sigma_{g-1}, \overline{\sigma_0},\ldots,\overline{\sigma_{g-1}}$, where $\sigma_i = \psi^i$ and $\overline{\sigma_i} = \iota \circ \psi^i$. It will be useful to use the notation $\sigma_i$ also for $i \geq g$, interpreting indices modulo $g$. Fix two integers $r, h$ smaller than $g$ (to be chosen below; we will only consider $r=h$, but other choices would also work).
We consider the two CM types $\Phi_1, \Phi_2$ defined by the conditions
\[
\Phi_1^* = \{ \overline{\sigma_i} : i =0,\ldots,r-1 \} \cup \{ \sigma_i : i = r,\ldots,g-1 \}
\]
and
\[
\Phi_2^* = \{\overline{\sigma_{hi}} : i=0,\ldots,r-1 \} \cup \{\sigma_{hi} : i =r,\ldots,g-1 \}.
\]
Notice that $\Phi \mapsto \Phi^*$ is an involution on the set of CM-types, so specifying $\Phi^*$ is equivalent to describing $\Phi$ itself.
One can check easily that choosing $r=h$ to be the smallest prime that does not divide $g$ yields CM types $\Phi_1, \Phi_2$ that are primitive and essentially different. Also notice that for these CM types the reflex field $E^*$ coincides with $E$.

By Theorem \ref{thm:CMExistence} there exist abelian varieties $A_1, A_2$, defined over some common number field $K$, such that $\operatorname{End}_{\overline{K}}(A_i) \otimes \mathbb{Q}=E$ and the CM type of $A_i$ is (essentially equal to) $\Phi_i$. By enlarging $K$ and replacing $A_i$ with some isogenous abelian variety if necessary, we can also assume that $\operatorname{End}_K(A_i)=\mathcal{O}_E$ (see \cite[Proposition 2.5.4]{CurvesOfGenus2}) and that $A_1, A_2$ have good reduction at all places of $K$ (Corollary \ref{cor:GoodReduction}). Since $\operatorname{End}_{\overline{K}}(A_i) \otimes \mathbb{Q} \cong E$ is a field, both $A_1$ and $A_2$ are geometrically simple, and since their CM types are essentially different we have $\operatorname{Hom}_{\overline{K}}(A_1, A_2)=(0)$ (Remark \ref{rmk:NotIsogenous}).
We now prove that condition \eqref{eq:Question} holds for our abelian varieties $A_1, A_2$ over a finite extension $K'$ of $K$. Notice that for each dimension $g$ we have constructed infinitely many geometrically non-isomorphic abelian varieties: indeed, their geometric endomorphism algebras fall into infinitely many distinct isomorphism classes. Since as already pointed out it suffices to check \eqref{eq:Question} in case $d$ is the power of a prime, the following result then proves Theorem \ref{thm:MainIntro} for $g \neq 4, 6$.

\begin{theorem}\label{thm:MainBody}
There is a finite extension $K'$ of $K$ with the following property: for every prime $\ell$ and every $n \in \mathbb{N} \cup \{\infty\}$ we have $K'(A_1[\ell^n]) = K'(A_2[\ell^n])$.
\end{theorem}
\begin{proof}
Denote by $\rho_1$, $\rho_2$ the $\ell$-adic Galois representations attached to $A_1, A_2$. By Theorem \ref{thm:CM} we may consider $\rho_i$ as a map $I_K \to (\mathcal{O}_E \otimes \mathbb{Z}_\ell)^\times$, and we get corresponding homomorphisms $\varepsilon_1, \varepsilon_2 : I_K \to \mathcal{O}_E^\times$. The open subgroup $K^\times  \cdot \left( \prod_{v \text{ finite}} U_v \times \prod_{v \text{ infinite}} K_v^\times \right)$ of $I_K$ defines a finite extension $K'$ of $K$. We restrict $\rho_1, \rho_2$ to $K^\times  \cdot \left( \prod_{v \text{ finite}} U_v \times \prod_{v \text{ infinite}} K_v^\times \right)$ or, which is the same, to $\operatorname{Gal}(\overline{K'}/K')$. 
The representation $\rho_i$ is trivial on $K^\times$ (because it factors via $\operatorname{Gal}(\overline{K}/K)^{\operatorname{ab}}$) and on $ \prod_{v \text{ infinite}} K_v^\times$, because for an idèle $a=(a_v)$ that is nontrivial only along the infinite components we have both $a_\ell=1$ and $\varepsilon(a)=1$,
 see Theorem \ref{thm:CM} and Remark \ref{rmk:EpsilonTrivialAtInfinity}. It follows that the representation $\rho_i|_{\operatorname{Gal}(\overline{K'}/K')}$ is completely determined by the restriction of $\rho_i$ to the subgroup $H:=\prod_{v \text{ finite}} U_v$.

Since the homomorphisms $\varepsilon_i$ are trivial on $U_v$ for each place $v$ of good reduction, and since by construction $A_i$ has good reduction at all finite places of $K$, by Theorem \ref{thm:CM} we see that for $a=(a_v) \in H$ we have
 \[
 \rho_1(a) = \prod_{\phi \in \Phi_1^*} \phi(N_{K/E}(a_\ell))^{-1} \quad \text{and} \quad  \rho_2(a) = \prod_{\phi \in \Phi_2^*} \phi(N_{K/E}(a_\ell))^{-1}.
 \]
Let $b_\ell=N_{K/E}(a_\ell)$ and for each integer $i$ let $\beta_i = \overline{\sigma_i(b_\ell)}/\sigma_i(b_\ell)$. We have $\psi \circ \sigma_i = \sigma_{i+1}$, and since $\iota$ commutes with $\psi$ we also have $\psi \circ \overline{\sigma_i} = \psi \circ \iota \circ \psi^i = \iota \circ \psi^{i+1} = \overline{\sigma_{i+1}}$. In particular, $\psi(\beta_i)=\beta_{i+1}$. Then
 \[
 \begin{aligned}
 \rho_2(a)^{-1} & = \prod_{\phi \in \Phi_2^*} \phi(b_\ell) =  \prod_{i=0}^{r-1} \beta_{hi} \cdot \prod_{i=0}^{g-1} \sigma_i(b_\ell) \\
 & = \prod_{i=0}^{r-1} \beta_{hi} \cdot \prod_{\sigma \in \operatorname{Gal}(E/F)} \sigma(b_\ell) = \prod_{i=0}^{r-1} \beta_{ri}  \cdot N_{E/F} \left( N_{K/E}(a_\ell) \right)\\
 & = \prod_{i=0}^{r-1} \beta_{hi} \cdot N_{K/F}(a_\ell).
 \end{aligned}
 \]
We remark that for simplicity of notation we are now writing the action of $\operatorname{Gal}(E/\mathbb{Q})$ on the left. This does not conflict with our previous conventions since $\operatorname{Gal}(E/\mathbb{Q})$ is abelian. 

Consider $\ker \rho_2|_H$: it consists of the idèles $a=(a_\ell) \in H$ such that $\prod_{i=0}^{r-1} \beta_{hi}  = N_{K/F}(a_\ell)^{-1}$. However, since $N_{K/F}(a_\ell)^{-1} \in F \otimes \mathbb{Q}_\ell$ is invariant under $\psi \in \operatorname{Gal}(E/F)$, applying $\psi^{h}$ to the previous equality we find 
\[
(a_\ell) \in \ker \rho_2|_H \Longleftrightarrow \prod_{i=0}^{r-1} \beta_{hi}  = N_{K/F}(a_\ell)^{-1} \Longleftrightarrow \prod_{i=0}^{r-1} \beta_{hi} =\prod_{i=1}^{r} \beta_{hi} =N_{K/F}(a_\ell)^{-1}.
\]
As $\beta_i$ is invertible for all $i$ (being the norm of an idèle), the last equivalent condition above implies $\beta_{hr}=\beta_0$. By repeatedly applying $\psi^{hr}$ we then obtain $\beta_{hrk} = \beta_0$ for all integers $k$. As $hr$ is assumed to be prime with $g$, this implies that all $\beta_i$ are equal, and since $\beta_i = \psi^i(\beta_0)$ this condition is also equivalent to $\psi(\beta_0)=\beta_0$.
We have thus proved that an idèle $(a_\ell) \in H$ is in $\ker \rho_2|_H$ if and only if the two conditions
 $\psi(\beta_0)=\beta_0$ and
$\beta_0^r = N_{K/F}(a_\ell)^{-1}$ hold. The same exact argument applies to $\rho_1|_H$ (it is simply the case $h=1$ of the above), which shows that $\rho_1|_H$ and $\rho_2|_H$ have the same kernel. 
We have then showed that the kernels of the $\ell$-adic representations attached to $A_1, A_2$ over $K'$ are the same, hence that $K'(A_1[\ell^\infty]) = K'(A_2[\ell^\infty])$.
Finally, the same argument works verbatim for modulo-$\ell^n$ representations (simply replace equalities with congruences modulo $\ell^n$), so the equality $K'(A_1[\ell^n]) = K'(A_2[\ell^n])$ holds for all primes $\ell$ and all positive integers $n$.
\end{proof}

\subsubsection{Small dimensions}

We now analyse the remaining cases $g=1, 2, 3, 4, 6$. Theorem \ref{thm:IntroClassification} completely settles the case when $A_1, A_2$ have dimension at most 2 (independently of whether they have complex multiplication or not). %
Examples for $g=4, 6$ (the cases missing to complete the proof of Theorem \ref{thm:MainIntro}) can be obtained as follows. Let $E/\mathbb{Q}$ be a CM field having Galois group isomorphic to $\mathbb{Z}/2g\mathbb{Z}$ (with generator $\sigma$), and consider the CM types
\[
\Phi_1 = \{ \sigma^2, \sigma^3, \sigma^4, \sigma^5 \}
, \quad 
\Phi_2 = \{ \sigma^0, \sigma^2, \sigma^5, \sigma^7 \} \quad \text{ if } g=4
\]
\[
\Phi_1 = \{ \sigma^0, \sigma, \sigma^2, \sigma^9, \sigma^{10}, \sigma^{11} \}
, \quad 
\Phi_2 = \{
\sigma^0, \sigma^2, \sigma^5 , \sigma^7, \sigma^9,  \sigma^{10}
\} \quad \text{ if } g=6.
\]
Let moreover $A_1, A_2$ be CM abelian varieties (defined over a suitable number field) with CM by $E$ and by the CM types $\Phi_1, \Phi_2$ respectively.
\begin{remark}
The above CM types have been found by looping through all possible CM types on $\mathbb{Z}/2g\mathbb{Z}$ and using the algorithm of Lemma \ref{lemma:MTIsoLinearAlgebraExplicit} to check when the projections $\MT(A_1 \times A_2) \to \MT(A_i)$ are isomorphisms, see \cite{Script}.
\end{remark}

We proceed as in the proof of Theorem \ref{thm:MainBody} to define the field $K'$. One can then check that $A_1, A_2$ are strongly iso-Kummerian over $K'$. For example, for $g=6$ (which is the slightly harder case), Theorem \ref{thm:CM} shows that $(a_v) \in \prod_{v \text{ finite}} U_v$ is in the kernel of the modulo-$\ell^n$ representation $\rho_i$ attached to $A_i$ if and only if $e_i \equiv 1 \pmod{\ell^n}$, where
\[
e_1 := b_\ell  \sigma^{11}(b_\ell)  \sigma^{10}(b_\ell)  \sigma^3(b_\ell)  \sigma^2(b_\ell) \sigma(b_\ell), \text{ resp. } e_2 := b_\ell \sigma^{10}(b_\ell) \sigma^7(b_\ell) \sigma^5(b_\ell) \sigma^3(b_\ell) \sigma^2(b_\ell)
\]
and $b_\ell=N_{K/E}(a_\ell)^{-1}$. Straightforward linear algebra leads to the identities $e_1 = \frac{\sigma^8(e_2)\sigma^9(e_2)\sigma^{10}(e_2)}{ \sigma^5(e_2) \sigma^6(e_2)}$ and $e_2=\frac{\sigma^5(e_1)\sigma^7(e_1)\sigma^9(e_1)\sigma^{11}(e_1)}{\sigma^6(e_1)\sigma^8(e_1)\sigma^{10}(e_1)}$, which show that $e_1$ is congruent to $1$ if and only if $e_2$ is.
As in the proof of Theorem \ref{thm:MainBody}, this implies that $A_1, A_2$ are strongly iso-Kummerian over $K'$. Computations for $g=4$ are analogous but simpler, and we omit them.
Finally, since one can easily check that $\Phi_1, \Phi_2$ are simple and essentially different, this gives examples of 4- and 6-dimensional absolutely simple abelian varieties that are strongly iso-Kummerian but geometrically non-isogenous. Notice that in this way we may construct infinitely many geometrically distinct examples: indeed, one can find CM fields with cyclic Galois group of order $2g$ by considering the unique subfield of $\mathbb{Q}(\zeta_p)$ of degree $2g$ for the (infinitely many) primes $p$ such that $p-1 \equiv 2g \pmod{4g}$.

It remains to consider the case $g=3$. %
Recall that an absolutely simple abelian threefold $A$ is of type IV in Albert's classification if $\operatorname{End}_{\overline{K}}(A) \otimes \mathbb{Q}$ is either a quadratic imaginary field or a sextic CM field (in the latter case $A$ is a CM abelian variety).

\begin{proposition}\label{prop:g3}
Let $K$ be a number field. If $A_1, A_2/K$ are absolutely simple, strongly iso-Kummerian abelian threefolds, then they are geometrically isogenous. %
\end{proposition}
\begin{proof}
Let $F_i := \operatorname{End}_{\overline{K}}(A_i) \otimes \mathbb{Q}$. The proof is divided into several cases:

\begin{enumerate}
\item at least one of $A_1, A_2$ is not of type IV in Albert's classification.%
\item both $A_1, A_2$ are of type IV in Albert's classification, and at least one of them has CM.%
\item both $A_1$ and $A_2$ are of type IV, and neither has CM.%
\end{enumerate}%

We recall that abelian threefolds satisfy the Mumford-Tate conjecture \cite{Ribet83classeson}, \cite[Theorem 6.1]{MR3494170}. We will also need the following facts: the Mumford-Tate group of an abelian variety $A$ is a torus if and only if $A$ has CM \cite[§5.3]{MTGps}; the Hodge group of a simple abelian threefold $A$ has nontrivial centre if and only if $A$ is of type IV in the sense of Albert \cite[§2.3]{MZ}; if $\operatorname{End}_{\overline{K}}(A)$ is an order in a quadratic imaginary field $F$, then $\MT(A)$ is isomorphic to the unitary group of $V := H_1(A(\mathbb{C}), \mathbb{Q})$, considered as an $F$-vector space of dimension $3$ endowed with a suitable $F$-Hermitian form $\psi$ \cite[§2.3]{MZ}. We will denote this group by $U_F(V, \psi)$. Finally, given a reductive group $G$ we denote by $G^{\operatorname{ss}}$ its derived (semisimple) subgroup.
\begin{enumerate}[wide, labelwidth=!, labelindent=0pt]
\item If neither $A_1$ nor $A_2$ are of type IV, then (since both $A_1$ and $A_2$ have odd dimension, hence odd relative dimension) Theorem \ref{thm:IntroClassification} shows that $A_1, A_2$ are geometrically isogenous. If only one of the two is of type IV, say $A_1$, then $\Hg(A_2)$ is semisimple, while the centre of $\Hg(A_1)$ is nontrivial. It follows that $\dim \Hg(A_1 \times A_2) = \dim \Hg(A_1 \times A_2)^{\operatorname{ss}} + \dim Z(\Hg(A_1 \times A_2)) \geq \dim \Hg(A_2)^{\operatorname{ss}} + \dim Z(\Hg(A_1)) \geq \dim \Hg(A_2) +1$, which contradicts Lemma \ref{lemma:IsoKummerianImplesSameHodgeDimension}.

\item By symmetry we may assume that $A_1$ has CM. If $A_2$ does not have CM, then its Mumford-Tate group is not a torus, so $\dim \Hg(A_2)^{\operatorname{ss}} > 0$. Recalling that $\Hg(A_1)$ is a torus, so that $\dim \Hg(A_1) = \dim Z(\Hg(A_1))$, as in part (1) it follows that $\dim \Hg(A_1 \times A_2) > \dim\Hg(A_1)$, contradiction.
Hence both $A_1$ and $A_2$ are CM, and the conclusion follows from Proposition \ref{prop:CMThreefolds}.
\item Suppose now that $A_1, A_2$ are both of type IV, but neither has CM. Then as recalled above $F_i$ is a quadratic imaginary field and we have $\MT(A_i) = U_{F_i}(V_i, \psi_i)$, where $V_i=H_1(A_i, \mathbb{Q})$ and $\psi_i$ is a suitable $F_i$-Hermitian form. By Lemma \ref{lemma:IsoKummerianImplesSameHodgeDimension}, the projection $\MT(A_1 \times A_2) \to \MT(A_i)$ is an isogeny, so the central tori of $\MT(A_1) \cong U_{F_1}(V_1, \psi_1)$ and $\MT(A_2) \cong U_{F_2}(V_2, \psi_2)$ are isogenous. Since these tori are $T_{F_1}, T_{F_2}$ respectively, one sees easily that $F_1=F_2$.
Let $\mathfrak{g}$ be the Lie algebra of $\MT(A_1 \times A_2)$. %
Using again the fact that $\MT(A_1 \times A_2) \to \MT(A_i)$ is an isogeny, we have
\[
\mathfrak{g}_{\mathbb{C}} \cong \operatorname{Lie}(\MT(A_i)) \otimes \mathbb{C} \cong\operatorname{Lie}(U_{F_i}(V_i, \psi_i)) \otimes \mathbb{C} \cong  \mathbb{C}^2 \oplus \mathfrak{sl}_3,
\]
where $\mathbb{C}^2$ can be identified with $F \otimes_{\mathbb{Q}} \mathbb{C}$ (here $F=F_1=F_2$) via 
\[
\begin{array}{ccc}
F \otimes_\mathbb{Q} \mathbb{C} & \to & \mathbb{C}^2 \\
f \otimes z & \mapsto & (zf, z \overline{f}).
\end{array}
\]
The structure of the cohomology groups $V_i \otimes \mathbb{C}$ (which are 6-dimensional $\mathbb{C}$-vector spaces) as representations of $\mathfrak{g}_\mathbb{C}$ is given in \cite[Proof of Theorem 3]{Ribet83classeson}: it is the direct sum of $\chi_1 \otimes \operatorname{Std}$ and $\chi_{2} \otimes \operatorname{Std}$, where $\operatorname{Std}$ is the standard representation of $\mathfrak{sl}_3$ and $\chi_j : \mathbb{C}^2 \to \mathbb{C}$ is the projection on the $j$-th factor ($j=1,2$). 
In particular, we have $V_1 \otimes \mathbb{C} \cong V_2 \otimes \mathbb{C}$ as representations of $\mathfrak{g}_\mathbb{C}$. %
Using the fact that $\MT(A_1 \times A_2)$ is connected, hence that (for any representation) the invariants for $\MT(A_1 \times A_2)$ and for $\mathfrak{g}$ coincide, we obtain
\[
\begin{aligned}
\dim_{\mathbb{Q}} \operatorname{Hom}_{\mathbb{C}}(A_1,A_2) \otimes \mathbb{Q} 
& = \dim_{\mathbb{Q}} \operatorname{Hom}_{\MT(A_1 \times A_2)}( V_1, V_2 ) \\
& = \dim_{\mathbb{Q}} \operatorname{Hom}_{ \mathfrak{g} }( V_1, V_2 ) \\
& = \dim_{\mathbb{C}} \operatorname{Hom}_{ \mathfrak{g}_{\mathbb{C}} }( V_1 \otimes \mathbb{C}, V_2 \otimes \mathbb{C} ) >0,
\end{aligned}
\]
because the two $\mathfrak{g}_\mathbb{C}$-representations $V_1 \otimes \mathbb{C}$ and $ V_2 \otimes \mathbb{C}$ are isomorphic. Since $A_1, A_2$ are simple, any non-zero map between them is an isogeny, and we are done.
\end{enumerate}
\end{proof}

\begin{corollary}\label{cor:DimensionAtMost3}
Let $K$ be a number field and let $A_1, A_2/K$ be strongly iso-Kummerian abelian varieties with $\dim(A_1) \leq 3, \dim(A_2) \leq 3$. Then $A_1$ and $A_2$ have the same isogeny factors (possibly with different multiplicities) over $\overline{K}$.
\end{corollary}
\begin{proof}
If all the absolutely simple isogeny factors of $A_1, A_2$ have dimension at most 2 we are done by Theorem \ref{thm:IntroClassification}. Thus up to symmetry we may assume that $A_1$ is absolutely simple of dimension $3$. It now suffices to show that $A_2$ is also absolutely simple of dimension 3, because the result will then follow from the previous proposition. By contradiction, let $B$ be an absolutely simple factor of $A_2$ of dimension at most 2. By \cite[§5.4]{MZ} we have $\Hg(A_1\times B)=\Hg(A_1) \times \Hg(B)$, unless $B$ is a CM elliptic curve with CM by the quadratic imaginary field $F$ and $F$ embeds in $\operatorname{End}^0_{\overline{K}}(A_1)$. The case $\Hg(A_1\times B)=\Hg(A_1) \times \Hg(B)$ is impossible by Lemma \ref{lemma:IsoKummerianImplesSameHodgeDimension}. In the remaining case, \textit{every} geometrically simple isogeny factor of $A_2$ is an elliptic curve with CM by the same quadratic imaginary field, and therefore $A_2$ is geometrically isogenous to $E^3$ for a suitable CM elliptic curve $E$. In this case, $\dim \Hg(A_2)=\dim \Hg(E^3)=1$, while $\dim \Hg(A_1) > 1$ by \cite{MZ}, again contradicting Lemma \ref{lemma:IsoKummerianImplesSameHodgeDimension} and finishing the proof.
\end{proof}

\subsubsection{An example of Shioda}\label{sect:Shioda}

We conclude our discussion of strongly iso-Kummerian CM abelian varieties by constructing an explicit strongly iso-Kummerian pair defined over $\mathbb{Q}$. This example was first noticed by Shioda \cite{Shioda81algebraiccycles} in a slightly different context, and was also considered by Zywina \cite[§1.7]{ZywinaMonodromy}.

\begin{theorem}\label{thm:Shioda}
Let $J/\mathbb{Q}$ be the Jacobian of the smooth projective curve $C$ with hyperelliptic equation $y^2=x^9-1$. Then $J$ is isogenous over $\mathbb{Q}$ to $X \times E$, where $X$ is an absolutely simple abelian threefold with CM by $\mathbb{Q}(\zeta_9)$ and $E$ is the CM elliptic curve $y^2=x^3-1$. The abelian varieties $X \times E$ and $X$ are strongly iso-Kummerian over $\mathbb{Q}$.
\end{theorem}

The proof of Theorem \ref{thm:Shioda} will occupy the rest of this section. We start by describing the abelian variety $X$.
The map $\pi : C \to E$ given by $(x,y) \mapsto (x^3,y)$ shows that $E$ appears as an isogeny factor of $J$. We let $X \subseteq J$ be the identity component of the kernel of $\pi_* : J \to E$: it is a 3-dimensional abelian subvariety of $J$, absolutely simple by results of Shioda \cite[Example 6.1]{Shioda81algebraiccycles}.
We may identify the complex uniformisation of $J$ with $H^0(C_{\mathbb{C}}, \Omega^1_C) = \mathbb{C} \langle \frac{dx}{2y}, x\frac{dx}{2y}, x^2\frac{dx}{2y}, x^3\frac{dx}{2y} \rangle$.
Pulling back regular differentials shows that $\pi^*E$ corresponds to the vector subspace $\mathbb{C} \langle \frac{x^2 \, dx}{2y} \rangle$ of $H^0(C_{\mathbb{C}}, \Omega^1_C)$, so $J/E$ (which is isogenous to $X$) has natural analytic uniformisation given by the quotient
\[
\mathbb{C} \langle \frac{dx}{2y}, x\frac{dx}{2y}, x^2\frac{dx}{2y}, x^3\frac{dx}{2y} \rangle / \mathbb{C} \langle x^2 \frac{dx}{2y} \rangle = \mathbb{C} \langle \frac{dx}{2y}, x\frac{dx}{2y}, x^3\frac{dx}{2y} \rangle.
\]
The automorphism $\alpha : (x,y) \mapsto (\zeta_9 x, y)$ induces an action of $F := \mathbb{Q}(\zeta_9)$ on this vector space. Since $\alpha^* x^i \frac{ dx}{y} = \zeta_9^{1+i} \frac{dx}{y}$, the characters appearing in this representation are $\zeta_9 \mapsto \zeta_9^i$ for $i=1,2,4$. This identifies the CM type $\Phi = \{\sigma_1, \sigma_2, \sigma_4\}$ of $\operatorname{Gal}(\mathbb{Q}(\zeta_9)/\mathbb{Q}) = \left\{ \sigma_i : \zeta_9 \mapsto \zeta_9^i \right\}$, which is the CM type of both $X$ and $J/E$. The reflex field is $E$ itself, and the dual CM type is $\Phi^* = \{\sigma_1, \sigma_5, \sigma_7\}$.

It is clear that we have an action of $\mathbb{Z}[\zeta_9] = \mathcal{O}_F$ on $J$. We denote by $[\alpha] \in \operatorname{End}_F(J)$ the endomorphism corresponding to $\alpha \in \mathcal{O}_F$.
For every $\alpha \in \mathcal{O}_F$ we have $[\alpha](X) = X$ (indeed $X$ is the unique abelian subvariety of $J$ of dimension 3), so $[\alpha]$ induces an endomorphism of $X$, and it follows easily that $\operatorname{End}_F(X) = \mathcal{O}_F$. Moreover, as $\pi_* \pi^* : E \to E$ is multiplication by $\deg \pi = 3$, we obtain that the canonical isogeny $X \times E \to J$ given by addition has degree a power of 3. In particular, for every positive integer $d$ prime to 3 we have $J[d] \cong X[d] \oplus E[d]$.

We now prove that $A := X \times E$ and $X$ are strongly iso-Kummerian over $F$. Once this is done, it will be easy to deduce that the same statement also holds over $\mathbb{Q}$. We certainly have $F(X[d]) \subseteq F(A[d])$, so we only need to show the other inclusion. Since $F(A[d]) = F(X[d], E[d])$, it suffices to show that $F(E[d]) \subseteq F(X[d])$. By Galois theory, this amounts to proving that every $\sigma \in \abGal{F}$ that fixes $X[d]$ also fixes $E[d]$. By the Chinese Remainder Theorem, it suffices to restrict to the case of $d$ being a power $\ell^n$ of a prime $\ell$. Proposition 2.3 in \cite{MR819231} shows that the open subgroup $F^\times \cdot \left( \prod_{v \text{ finite}} U_v \times \prod_{v \text{ infinite}} F_v^\times \right)$ of $I_F$ corresponds to the abelianised absolute Galois group of the Hilbert class field of $F$. Since $F$ has class number 1, we find that this subgroup 
maps surjectively onto $\abGal{F}^{\operatorname{ab}}$. As in the proof of Theorem \ref{thm:MainBody}, the $\ell$-adic Galois representations attached to $X$ and $E$ over $F$ are then completely determined by their restrictions to $\prod_{v \text{ finite}} U_v$.
More precisely, these representations %
are described by Theorem \ref{thm:CM}: given a finite idèle $a=(a_v) \in \prod_{v} U_v$ 
we have
\[
\rho_{E, \ell^\infty}(a) = \varepsilon_E(a) N_{F/\mathbb{Q}(\zeta_3)} (a_\ell)^{-1}, \quad \rho_{X, \ell^\infty}(a) = \varepsilon_X(a) N_{\Phi^*} (a_\ell)^{-1},
\]
where we denote by $N_{\Phi^*}$ both the reflex norm $F^\times \to F^\times$ attached to the CM type of $X$ described above and its extension to $(F \otimes \mathbb{Q}_\ell)^\times \to (F \otimes \mathbb{Q}_\ell)^\times$. Here $\varepsilon_E$ and $\varepsilon_X$ are characters of $\prod_{v} U_v$ with values in the roots of unity in $\mathbb{Q}(\zeta_3)$ and $F$, respectively. Since it is easily checked that $C$ has good reduction away from 2 and 3, it follows from Theorem \ref{thm:CM} that $\varepsilon_X$ and $\varepsilon_E$ factor via the product $\prod_{w \mid 6}U_w$ of local units $U_w$ at the places $w$ of $\mathcal{O}_F$ above $2$ and $3$. We now prove that $\varepsilon_E(a)=\varepsilon_X(a)^3$, or equivalently $\varepsilon_E(a_w)=\varepsilon_X(a_w)^3$ for every place $w$ of bad reduction.

In order to do this, recall from \cite[p.~508]{MR0236190} that the restriction of $\varepsilon_X$ to $U_w$ coincides with the homomorphism $\varphi_w$ described in the proof of \cite[Theorem 6 (b)]{MR0236190}. 
This homomorphism is obtained as follows: let $L/F$ be  a finite extension over which $X$ acquires good reduction at the places above $w$. The homomorphism $\varphi_w$ is then given by the action of inertia on the special fibre of the Néron model of the abelian variety $X/L$. Observe that, since all the automorphisms of $J$ send $X$ to itself, one can equivalently consider the action of inertia on the Néron model of $J$ instead. We now simplify this description by using the fact that $J$ is a Jacobian.

An element $\sigma$ in the inertia at $w$ acts on the Néron model $\mathcal{J}_L$ of $J/L$ via the finite group $\operatorname{Aut}(\mathcal{J}_L) =\operatorname{Aut}(J_F) =\mu(F)$. Since this group coincides with the automorphism group of the curve $C$, we may equivalently describe the action of $\sigma$ on $\tilde{C}$, and obtain $\varphi_{J,w}(\sigma)$ as follows. We take as $L$ a finite extension of $F$ over which $C$ admits a semistable model $\tilde{C}$ (in particular, $J/L$ is semistable and hence has good reduction everywhere), and fix an isomorphism $\psi : C \to \tilde{C}$ defined over $L$. Then $\psi^{-1} \circ {}^\sigma \psi$ is an automorphism of $C$, which is precisely $\varphi_{J,w}(\sigma)$ as it can be seen by embedding $C$ into its Jacobian.
Since $J$ has good reduction over $L$, so do its subvariety $X$ and its quotient $E$. Thus we may use the same extension $L$ in order to compute the automorphisms of $X$ and $E$ induced by $\sigma$; these will be our $\varepsilon_X(\sigma)$ and $\varepsilon_E(\sigma)$. To describe them, we first notice that we have already proved that the restriction map $\mathcal{O}_F = \operatorname{End}(J) \to \operatorname{End}(X)=\mathcal{O}_F$ is the identity. Thus $\varepsilon_{X}(\sigma)=\varphi_{J,w}(\sigma)$. On the other hand, since the map $\pi : C \to E$ is $(x,y) \mapsto (x^3,y)$, the map $\mu(F) = \operatorname{Aut}_F(J) \to \operatorname{Aut}_F(E) = \mu(\mathbb{Q}(\zeta_3))$ is easily seen to be $\zeta \mapsto \zeta^3$. Thus $\varepsilon_E(\sigma)=\varphi_{J,w}(\sigma)^3$, and we conclude, as desired, that $\varepsilon_E(a_w)=\varepsilon_X(a_w)^3$ for all $a_w \in U_w$.

We can finally prove the inclusion $F(E[\ell^n]) \subseteq F(X[\ell^n])$. Indeed, let $a \in \prod_{v \text{ finite}} U_v$ be an element whose image in $\abGal{F}^{\operatorname{ab}}$ fixes $X[\ell^n]$. By what we already proved, this amounts to the condition
\begin{equation}\label{eq:TrivialModElln}
\varepsilon_X(a_3) \sigma_1(a_\ell)^{-1}\sigma_5(a_\ell)^{-1}\sigma_7(a_\ell)^{-1} \equiv 1 \pmod{\ell^n}.
\end{equation}
Applying $\sigma_5$ to both sides of this congruence we find
$
\varepsilon_X(a_3)^5 \sigma_5(a_\ell)^{-1}\sigma_7(a_\ell)^{-1}\sigma_{-1}(a_\ell)^{-1} \equiv 1 \pmod{\ell^n}$, and comparing with Equation \eqref{eq:TrivialModElln} we obtain $
\varepsilon_X(a_3)^5\sigma_{-1}(a_\ell)^{-1} \equiv \varepsilon_X(a_3) \sigma_1(a_\ell)^{-1} \pmod{\ell^n}$. Applying $\sigma_5$ again we get $
\varepsilon_X(a_3)^{20}\sigma_{4}(a_\ell)^{-1} \equiv \sigma_5(a_\ell)^{-1} \pmod{\ell^n}$, and replacing this congruence in Equation \eqref{eq:TrivialModElln} we finally obtain
\[
\varepsilon_X(a_3)^{21} \sigma_1(a_\ell)^{-1}\sigma_4(a_\ell)^{-1}\sigma_7(a_\ell)^{-1} \equiv 1 \pmod{\ell^n}.
\]
As $\varepsilon_X(a_3)$ is an $18$-th root of unity we have $\varepsilon_X(a_3)^{21}=\varepsilon_X(a_3)^{3}=\varepsilon_E(a_3)$, and therefore 
\[
1 \equiv \varepsilon_X(a_3)^{21} \sigma_1(a_\ell)^{-1}\sigma_4(a_\ell)^{-1}\sigma_7(a_\ell)^{-1} \equiv  \varepsilon_E(a_3) N_{F/\mathbb{Q}(\zeta_3)}(a_\ell)^{-1} = \rho_{E, \ell^n}(a) \pmod{\ell^n},
\]
hence $a$ also acts trivially on $E[\ell^n]$. By Galois correspondence, this shows that $F(E[\ell^n]) \subseteq F(A[\ell^n])$ for all $\ell$ and $n$, and therefore that $A$ and $X$ are strongly iso-Kummerian over $F$. 

It remains to show that the same holds over $\mathbb{Q}$. As above, it suffices to prove that $\mathbb{Q}(E[\ell^n]) \subseteq \mathbb{Q}(X[\ell^n])$. By a theorem of Silverberg \cite{MR1154704}, for $\ell^n \geq 3$ the $\ell^n$-torsion field $\mathbb{Q}(X[\ell^n])$ contains the field of definition of the endomorphisms of $X$, which is $F$ (notice that by Theorem \ref{thm:CM} every field of definition of the endomorphisms contains the reflex field of the CM type of $X$). Hence we obtain $\mathbb{Q}(X[\ell^n]) = F(X[\ell^n]) \supseteq F(E[\ell^n]) \supseteq \mathbb{Q}(E[\ell^n])$ as desired.
Finally we consider the case $\ell^n=2$. Since the isogeny $E \times X \sim J$ has odd degree we have $J[2] \cong E[2] \oplus X[2]$. By well-known properties of hyperelliptic Jacobians, the $2$-torsion $J[2]$ is the $\mathbb{F}_2$-vector space spanned by the divisor classes $e_i = [( \zeta_9^i , 0 ) - (\infty)]$ for $i=0,\ldots,8$, subject to the only relation $\sum_{i=0}^8 e_i = 0$. Similarly, $E[2]$ is spanned by $f_i = [ ( \zeta_3^i ,0) - (\infty) ]$ for $i=0,1,2$, with the only relation $\sum_{i=0}^2 f_i=0$. The pullback $\pi^* : E \to J$ sends $f_i$ to $e_i + e_{i+3} + e_{i+6}$. Thus $X[2]=J[2]/E[2]$ is free with basis the classes of $e_0,\ldots,e_5$, and from this one checks that $\mathbb{Q}(E[2])$ is $\mathbb{Q}(\zeta_3)$, while $\mathbb{Q}(X[2])$ is $F$. Thus also for $\ell^n=2$ we have $\mathbb{Q}(E[\ell^n]) \subseteq \mathbb{Q}(X[\ell^n])$, which concludes the proof of Theorem \ref{thm:Shioda}.

\begin{remark}
The projection $\MT(X \times E) \to \MT(X)$ is an isomorphism. This fact can be deduced from the above proof, but we will show it using the algorithm suggested by Lemmas \ref{lemma:MTIsoLinearAlgebra} and \ref{lemma:MTIsoLinearAlgebraExplicit}. In the notation of those lemmas, we can take $E_1=F, E_2=\mathbb{Q}(\zeta_3)$ and $L=F$. The character group $\widehat{T}_F$ has natural basis $\{\sigma_i : (i,9)=1\}$, while $\widehat{T}_{E_2}$ has basis $\{\sigma, \overline{\sigma}\}$ given by the two complex conjugate embeddings of $\mathbb{Q}(\zeta_3)$ into $\mathbb{C}$. The morphism $N_1^*$ induced on character groups by the norm $N_1$ is the identity, while $N_2^*$ sends $\sigma$ to $\sigma_1+\sigma_4+\sigma_7$ and $\overline{\sigma}$ to $\sigma_2+\sigma_5+\sigma_8$. Concerning the reflex norms, we have that $\phi_2^*$ is the identity, while $\phi_1^* : \widehat{T}_L=\mathbb{Z}[\sigma_i] \to \widehat{T}_L = \mathbb{Z}[\sigma_i]$ is multiplication by $\sigma_1+\sigma_5+\sigma_7$. Thus $N_1^*\phi_1^* + N_2^*\phi_2^*$ is represented by the matrix
\[
\begin{pmatrix}
1 & 1 & 1 & 0 & 0 & 0 & 1 & 0 \\
0 & 1 & 1 & 0 & 0 & 1 & 0 & 1 \\
0 & 0 & 1 & 0 & 1 & 1 & 1 & 0 \\
1 & 1 & 0 & 1 & 0 & 0 & 0 & 1 \\
1 & 0 & 0 & 1 & 1 & 0 & 1 & 0 \\
0 & 0 & 0 & 1 & 1 & 1 & 0 & 1
\end{pmatrix}.
\]
Here the first six columns are indexed by the $\sigma_i$ with $(i,9)=1$, and the last two by $\sigma, \overline{\sigma}$. The rows are indexed by the $\sigma_i$ with $(i,9)=1$. The column vectors
$(0,  0,  1,  0,  1, -1, -1,  0)^T$ and $(0, 0, 0,  1, -1,  1,  0, -1)^T$ are in the (right) kernel of this matrix, and together with $\widehat{T}_F \times \{0\}$ they clearly generate $\widehat{T}_F \times \widehat{T}_{\mathbb{Q}(\zeta_3)}$. By Lemma \ref{lemma:MTIsoLinearAlgebra}, this proves the claim.
\end{remark}

\section{Abelian varieties with $\operatorname{End}(A_{\overline{K}})=\mathbb{Z}$}\label{sect:EndZ}
We now consider the other end of the spectrum with respect to the CM case, namely, abelian varieties with trivial geometric endomorphism ring.
We will show that there exist abelian varieties $A, B$ with $\operatorname{End}_{\overline{K}}(A) = \operatorname{End}_{\overline{K}}(B) = \mathbb{Z}$ and without common isogeny factors for which the canonical projections $\Hg(A \times B) \to \Hg(A), \Hg(A \times B) \to \Hg(B)$ both have finite kernel. We are unfortunately unable to construct examples for which the kernels in question are actually trivial (see also Remark \ref{rmk:ParityObstruction}).
Our main tool is the Kuga-Satake construction, introduced in \cite{MR0210717}, and for which we refer the reader to \cite{vG}. See also \cite{Moonen_noteson} for the notion of Mumford-Tate group in the more general context of Hodge structure of arbitrary weight.

\medskip

Let $n$ be a fixed positive integer and $V$ be a $\mathbb{Q}$-vector space of dimension $n+2$. Fix a quadratic form $Q$ on $V$ of signature $(2-,n+)$. By \cite[Lemma 4.7]{vG} we can equip $V$ with a Hodge structure of weight 2 in such a way that $\dim V^{2,0}=1$ and $\MT(V)=\GO(Q)$, the group of orthogonal similitudes of $Q$. Let $h:\mathbb{S} \to \operatorname{GL}_{V_\mathbb{R}}$ be the morphism defining any such Hodge structure on $V$, where $\mathbb{S}=\operatorname{Res}_{\mathbb{C}/\mathbb{R}}(\mathbb{R})$ is the Deligne torus.

Let $C(Q)$ be the Clifford algebra over $Q$ and $C^+(Q)$ be its even part. The Kuga-Satake construction takes as input the data $(V,h,Q)$, and gives as output a \textit{polarisable} Hodge structure on the rational vector space $C^+(Q)$ of weight 1 and type $(1,0)\oplus(0,1)$, which we will denote by $(C^+(Q),h_s)$. By passing to the dual, such a Hodge structure identifies a complex abelian variety $A$ (up to isogeny): indeed, the category of complex abelian varieties (up to isogeny) and that of polarisable Hodge structures of type $(-1,0)\oplus(0,-1)$ are equivalent (see \cite[§8]{vG} for details).

Recall that $V$ is canonically a vector subspace of $C(Q)$, and that with this identification the group $\CSpin(Q)$ is defined as 
\[
\CSpin(Q)(R)=\left\{ g \in (C^+(Q) \otimes R)^* \bigm\vert g(V \otimes 1) g^{-1} \subseteq V \otimes 1 \right\}
\]
for every $\mathbb{Q}$-algebra $R$. %
For the general properties of the Spin group, the reader can refer to \cite[pp.~304-308]{FH} and \cite{MR1701598}. By its very definition, $\CSpin(Q)$ admits a representation
\begin{equation}\label{eq:CSpinRep}
\begin{matrix}
\rho: & \CSpin(Q) & \to & \operatorname{GL}_{C^+(Q)} \\
      &    g     & \mapsto & (x \mapsto gx),
\end{matrix}
\end{equation}
and it can be checked that $\operatorname{End}_{\CSpin(Q)} \left(C^+(Q) \right) \cong C^+(Q)$, the isomorphism being induced by the right action of $C^+(Q)$ on itself. Moreover, one has $\MT(C^+(Q),h_s)=\CSpin(Q)$ as soon as $\MT(V)=\GO(Q)$ \cite[Proposition 6.3]{vG}, and clearly the dual Hodge structure also has Mumford-Tate group $\GO(Q)$.
We will need to know that there is a certain \textit{spinor norm} $\operatorname{CSpin}(Q) \to \mathbb{G}_{m,\mathbb{Q}}$, and that the Spin group $\Spin(Q)$ can be described both as the derived subgroup of $\CSpin(Q)$ and as the kernel of the spinor norm. It is also easy to see the $\CSpin(Q)$ is the almost-direct product of $\Spin(Q)$ and $\mathbb{G}_m$.

We now turn to the study of the abelian variety $A$ obtained from the Kuga-Satake construction. Its dimension is
\[
\dim A = \frac{1}{2} \dim_{\mathbb{Q}} C^+(Q) = \frac{1}{4} \dim_{\mathbb{Q}} C(Q) = 2^n.
\]
The endomorphism algebra of $A$ is canonically identified with 
\[
\operatorname{End}_{\MT(C^+(Q))}(C^+(Q)) \cong \operatorname{End}_{\CSpin(Q)}(C^+(Q)) \cong C^+(Q).
\]
Choose $n=4k+2$ and $Q = x_1^2+x_2^2+\ldots+x_{n}^2-x_{n+1}^2-x_{n+2}^2$. 
The structure of $C^+(Q)$ is as follows:
\begin{proposition}\label{prop_KS} Let $n=4k+2$ and $Q = x_1^2+x_2^2+\ldots+x_{n}^2-x_{n+1}^2-x_{n+2}^2$, and write $\mathbb{H}$ for the rational quaternions $\mathbb{Q} \oplus \mathbb{Q}i\oplus \mathbb{Q}j \oplus \mathbb{Q}k$ with $i^2=j^2=k^2=-1, ij=-ji=k$.
\begin{itemize}
\item If $k$ is odd we have $C^+(Q) \cong \operatorname{Mat}_{2^{2k}}(\mathbb{H}) \oplus \operatorname{Mat}_{2^{2k}}(\mathbb{H})$;
\item If $k$ is even we have $C^+(Q) \cong \operatorname{Mat}_{2^{2k+1}}(\mathbb{Q}) \oplus \operatorname{Mat}_{2^{2k+1}}(\mathbb{Q})$.
\end{itemize}

\end{proposition}

\begin{proof}
This is proved exactly as \cite[Theorem 7.7]{vG}, the only difference being that we keep track of the quaternion algebra denoted $D$ in op.~cit.
\end{proof}
Combining the Kuga-Satake construction with the above proposition we obtain a family of examples of abelian varieties $A, B$ for which the two projections $\Hg(A \times B) \to \Hg(A)$ and $\Hg(A \times B) \to \Hg(B)$ have finite kernel:

\begin{theorem}\label{thm_Counterexamples}
Let $k$ be a positive integer. There exist simple, non-isogenous $\mathbb{C}$-abelian varieties $A_+, A_-$ such that
\begin{itemize}
\item $\operatorname{End}^0(A_+) \cong \operatorname{End}^0(A_-) \cong \begin{cases} \mathbb{H}, \mbox{ if } k \mbox{ is odd}; \\ \mathbb{Q}, \mbox{ if } k \mbox{ is even} \end{cases}$
\item $\dim(A_+)=\dim(A_-)=\begin{cases} 2^{2k+1}, \mbox{ if } k \mbox{ is odd}; \\ 2^{2k}, \mbox{ if } k \mbox{ is even} \end{cases}$
\item $\Hg(A_+ \times A_-)$ is absolutely simple. The projections $\Hg(A_+ \times A_-) \to \Hg(A_+)$ and $\Hg(A_+ \times A_-) \to \Hg(A_-)$ have finite kernel of order 2.
\end{itemize}
\end{theorem}

\begin{proof}
Let $n=4k+2$, choose $V,Q,h$ as above and let $A$ be the $\mathbb{C}$-abelian variety up to isogeny attached to the triple $(V,Q,h)$ by the Kuga-Satake construction (that is, the abelian variety corresponding to the dual Hodge structure of $(C^+(Q),h_s)$). Combining the results above we see that $\dim(A)=2^{n}$ and
\[
\operatorname{End}^0(A) \cong C^+(Q) \cong \begin{cases} \operatorname{Mat}_{2^{2k}}(\mathbb{H}) \oplus \operatorname{Mat}_{2^{2k}}(\mathbb{H}), \mbox{ if } k \mbox{ is odd} \\ \operatorname{Mat}_{2^{2k+1}}(\mathbb{Q}) \oplus \operatorname{Mat}_{2^{2k+1}}(\mathbb{Q}),\mbox{ if } k \mbox{ is even}\end{cases}
\]
For the sake of brevity let $m=2k$ (resp.~$2k+1$) if $k$ is odd (resp.~even). From the description of its endomorphism algebra we obtain that $A$ is isogenous to $(A_+ \times A_-)^{2^m}$ for simple, non-isogenous Abelian varieties $A_+, A_-$ whose endomorphism algebra is $\mathbb{H}$ or $\mathbb{Q}$, depending on the parity of $k$. Moreover,
\[
\dim(A_+)=\dim(A_-)=\frac{1}{2} \frac{\dim(A)}{2^m} = 2^{n-m-1} = \begin{cases} 2^{2k+1}, \mbox{ if } k \mbox{ is odd} \\ 2^{2k}, \mbox{ if } k \mbox{ is even} \end{cases},
\]
because the decomposition $A_+ \times A_-$ comes from the analogous decomposition of the spin representation as a direct sum of the two non-isomorphic half spin representations \cite[p.~305]{FH}.
Note now that $\MT(A)$ is isomorphic to $\CSpin(Q)$, so $\Hg(A) \cong \Spin(Q)$ is absolutely simple (over the complex numbers, the group $\Spin(n+2)$ is simple as soon as $n+2>4$). The kernels of the surjective projections 
\[
\Hg \left( A_+ \times A_- \right) \to \Hg(A_+), \quad \Hg \left( A_+ \times A_- \right) \to \Hg(A_-)
\]
must therefore be finite as claimed. More precisely, from the discussion above we see that these kernels may be identified with the kernels of the half-spin representations of $\Spin(Q)$, hence they have order 2 \cite[III.6.1]{MR1636473}.
\end{proof}

\medskip

The above complex construction leads to the following version over number fields:
\begin{theorem}\label{thm_CounterHl}
Let $h$ be a positive integer. There exist geometrically simple abelian varieties $A_+, A_-$, defined over a number field $K$, non-isogenous over $\overline{K}$, and such that
\begin{itemize}
\item $\operatorname{End}_{\overline{K}}(A_+) \otimes \cong \operatorname{End}_{\overline{K}}(A_-) \otimes  \cong \mathbb{Z}$;
\item $\dim(A_+)=\dim(A_-)=2^{4h}$;
\item the Mumford-Tate conjecture holds for $A_+, A_-$ and $A_+ \times A_-$;
\item for every rational prime $\ell$ the group $H_\ell(A_+ \times A_-)$ is a double cover (via the canonical projections) of both $H_\ell(A_+)$ and $H_\ell(A_-)$.
\end{itemize}
\end{theorem}
\begin{proof}
Let $B_+, B_-$ be the abelian varieties whose existence is guaranteed by Theorem \ref{thm_Counterexamples} with $k=2h$. By Proposition \ref{prop:DescentToNumberFields}, there exist abelian varieties $A_+, A_-$, defined over a number field $K$, with the same Mumford-Tate groups as $B_+, B_-$ and such that the Mumford-Tate conjecture holds for $A_+, A_-$ and $A_+ \times A_-$. The morphism $H_\ell(A_+ \times A_-) \to H_\ell(A_+)$ may then be identified with the base-change to $\mathbb{Q}_\ell$ of the corresponding morphism $\Hg(A_+ \times A_-) \to \Hg(A_+)$, and similarly for $A_-$. As discussed in the proof of Theorem \ref{thm_Counterexamples}, this morphisms are precisely the half-spin representations of a certain Spin group and have kernel of order 2.
\end{proof}

The abelian varieties of the previous theorem are almost-examples of the situation considered in condition \eqref{eq:Question}. For simplicity we only study the case $d=\ell^\infty$, but the result could also be adapted to finite values of $d$.
\begin{corollary}\label{cor:AlmostStronglyIsokummerian}
Let $A_+, A_-$ be as in the previous theorem and write $A_0$ for the abelian variety $A_+ \times A_-$. There exists a finite extension $K'$ of $K$ such that for every prime $\ell$ we have
\[
[ K'(A_0[\ell^\infty]) : K'(A_+[\ell^\infty]) ] \leq 2, \quad [ K'(A_0[\ell^\infty]) : K'(A_-[\ell^\infty]) ] \leq 2.
\]
\end{corollary}
\begin{proof}
Up to replacing $K$ with a finite extension $K'$, we may assume that for each prime $\ell$ the image of the $\ell$-adic Galois representation attached to $A_+/K'$ (respectively $A_0/K'$) lands inside the $\mathbb{Q}_\ell$-points of $\MT(A_+)$ (respectively $\MT(A_0)$), see \cite[133]{MR3185222} or \cite{MR1441234}.
Let $K_{0} := K'(A_0[\ell^\infty])$ and $K_{+} := K'(A_+[\ell^\infty])$. Also let $\rho_{0, \ell^\infty}$ and $\rho_{+, \ell^\infty}$ be the $\ell$-adic Galois representations attached to $A_0/K'$ and $A_+/K'$ respectively, with images $G_{0, \ell^\infty}$ and $G_{+, \ell^\infty}$.
The extension $K_0 / K_+$ is Galois, with Galois group isomorphic to the kernel of the natural surjective map $G_{0, \ell^\infty} \to G_{+, \ell^\infty}$. By construction we have $G_{0, \ell^\infty} \subseteq G_\ell(A_0)(\mathbb{Q}_\ell)$ and $G_{+, \ell^\infty} \subseteq G_\ell(A_+)(\mathbb{Q}_\ell)$, and the projection map corresponds to the algebraic map $G_\ell(A_0) \to G_\ell(A_+)$. Since by Lemma \ref{lemma:kerCSpin} below the kernel of this map has order 2 (even over $\overline{\mathbb{Q}_\ell}$; note that we are working with the Spin group $\Spin(n+2)=\Spin(4h+4)$), the projection $G_{0, \ell^\infty} \to G_{+, \ell^\infty}$ has kernel of order at most 2, as claimed. The proof for $A_-$ is identical.
\end{proof}

\begin{lemma}\label{lemma:kerCSpin}
Let $F$ be an algebraically closed field of characteristic different from 2. Consider the $F$-algebraic groups $\operatorname{Spin}(4m)$ and $\operatorname{CSpin}(4m)$, defined from any non-degenerate quadratic form $Q$ on $F^{4m}$. The half-spin representations $\rho_{\pm}$ of $\Spin(4m)$ extend to representations of $\CSpin(4m) = \mathbb{G}_m \cdot \operatorname{Spin}(4m)$ by setting $\rho_{\pm}(\lambda g) = \lambda \rho_{\pm}(g)$. The kernels of these extended representations have order 2.
\end{lemma}
\begin{proof}
The fact that the half-spin representations extend is clear from their construction (see Equation \eqref{eq:CSpinRep}). It is well-known that the half-spin representations of $\Spin(4m)$ are orthogonal (see for example \cite[Table 2.1]{MR3494170}), so for $g \in \Spin(4m)$ the matrix $\rho_{\pm}(g)$ is a multiple of the identity if and only if it is $\pm \operatorname{Id}$ (the only orthogonal matrices that are also scalars are $\pm \operatorname{Id}$). Hence if $\rho_{\pm}(\lambda g) = \lambda \rho_{\pm}(g)$ is trivial for some $\lambda g \in \CSpin(4m)$ we must have $\lambda = \pm 1$, so $\lambda g$ is in $\Spin(4m)$. It follows that the kernels of the extended representations coincide with the kernels of the original representations of $\Spin(4m)$, which have order 2.
\end{proof}

\begin{remark}\label{rmk:ParityObstruction}
The example of Corollary \ref{cor:AlmostStronglyIsokummerian} is motivated on the one hand by the condition appearing in Lemma \ref{lemma:IsoKummerianImplesSameHodgeDimension}, and on the other by an examination of Table 2.1 in \cite{MR3494170}: the two half-spin representations of a Lie algebra of type $D_l$ provide the simplest example for which the criterion of \cite[Theorem 1.1]{MR3494170} fails. However, in order for $H_\ell(A_+ \times A_-) \to H_\ell(A_+)$ to be an isomorphism, one would need the half-spin representations to be faithful, which happens only if $l$ is odd (the present $l$ is connected to the parameter $n$ in the construction above by $l = (n+2)/2$). On the other hand, when $l$ is odd \cite[Table 2.1]{MR3494170} shows that the half-spin representations are not self-dual, complicating the construction of corresponding polarised Hodge structures.
\end{remark}

\bibliographystyle{abbrv}
\bibliography{Biblio}
\end{document}